\documentclass[12pt,a4paper,final]{article}       

\usepackage[margin=24.7mm]{geometry}

\usepackage{amsmath,amssymb,amsthm,xcolor}
\usepackage[notref,notcite]{showkeys}

\usepackage[utf8]{inputenc}

\usepackage{hyperref}

\newcommand{\bs}{\boldsymbol}
\newcommand{\diver}{\operatorname{div}}
\newcommand{\dd}{\,\mathrm{d}}
\newcommand{\CD}{C_{\mathfrak D}}
\newcommand{\weak}{\rightharpoonup}

\numberwithin{equation}{section}

\newtheorem{theorem}{Theorem}[section]
\newtheorem{proposition}[theorem]{Proposition}
\newtheorem{lemma}[theorem]{Lemma}
\newtheorem{corollary}[theorem]{Corollary}
\newtheorem{remark}[theorem]{Remark}
\newtheorem{definition}[theorem]{Definition}
\newcommand{\eps}{\varepsilon}
\newcommand{\bfV}{{\bs V}}
\newcommand{\bfH}{{\bs H}}
\newcommand{\rr}{{\overline s}}
\newcommand{\sss}{{\overline\tau}}

\newcommand{\boxlength}{a}
\newcommand{\INTinf}{{]0,\infty[}}
\newcommand{\INTell}{{]0,\boxlength[}}
\newcommand{\INTT}{{]0,T[}}
\renewcommand{\kappa}{\varkappa}

\definecolor{myred}{RGB}{255,0,40}
\definecolor{mygreen}{RGB}{30,150,30}
\definecolor{bothcolor}{RGB}{255,0,255}

\allowdisplaybreaks[4]
\begin{document}

\title{On the existence of global-in-time weak solutions\\
     and scaling laws for \\
     Kolmogorov's two-equation model for turbulence%
     \thanks{A.M. was partially supported by Deutsche
      Forschungsgemeinschaft (DFG) via  SFB\,910 \emph{Control of
      Self-Organizing Nonlinear Systems} (project no.\ 
    163436311), subproject A5 ``Pattern
    formation in coupled parabolic systems''.} 
}


\author{\large  Alexander Mielke$^{a,b}$ and\  Joachim Naumann$^b$\\[0.3em]
 {\normalsize \begin{tabular}{c@{\,}p{11cm}}
 $^a$ & Weierstra\ss{}-Institut f\"ur Angewandte Analysis und
    Stochastik\\
 $^b$& Institut f\"ur Mathematik, Humboldt-Universit\"at zu Berlin
 \end{tabular}}\\
 {\small \texttt{alexander.mielke@wias-berlin.de}, 
  \texttt{jnaumann@math.hu-berlin.de}} 
}



\date{Revision of 3. June 2022}

\maketitle

{\small\noindent\textbf{Abstract.}
  This paper is concerned with Kolmogorov's two-equation model for turbulence
  in $\mathbb{R}^3$ involving the mean velocity $\bs{u}$, the pressure $p$, an
  average frequency $\omega>0$, and a mean turbulent kinetic energy $k$. We
  consider the system with space-periodic boundary conditions in a cube
  $\Omega=\big(\INTell\big){}^3$, which is a good choice for studying the decay
  of free turbulent motion sufficiently far away from boundaries. In
  particular, this choice is compatible with the rich set of similarity
  transformations for turbulence.

  The main part of this work consists in proving existence of global
  weak solutions of this model. For this we approximate the system by
  adding a suitable regularizing $r$-Laplacian and invoke existence
  result for evolutionary equations with pseudo-monotone operators. An
  important point constitutes the derivation of pointwise a priori 
  estimates for $\omega$ (upper and lower) and $k$ (only lower) that
  are independent of the box size $\boxlength$, thus allow us to control the
  parabolicity of the diffusion operators.
}




\section{Introduction}\label{s1}

In 1942, A.N. Kolmogorov (see \cite{Ref14} and
\cite[pp.\,214--216]{Ref25} for an English translation) postulated the
following system of PDEs as a model for the isotropic homogeneous
turbulent motion of an incompressible fluid  $(x,t)\in \mathbb{R}^3\times
\INTinf$:
\begin{subequations}
\label{eq:I.1-4}
\begin{align}
 \label{1.1}
 \diver \bs{u}&=0 \;,
\\
 \label{1.2}
 \frac{\partial\bs{u}}{\partial t}+(\bs{u}\cdot\nabla)\bs{u}
  &=\nu_0\diver\Big(\frac{k}{\omega}\,\bs{D}(\bs{u})\Big)-\nabla
  p+\bs{f},
\\[1mm] 
 \label{1.3}
 \frac{\partial \omega}{\partial t}+\bs{u}\cdot\nabla\omega
  &= \nu_1 \diver\Big(\frac k \omega\,\nabla\omega\Big)- \alpha_1
  \omega^2,
\\[1mm]
 \label{1.4}
 \frac{\partial k}{\partial t}+\bs{u}\cdot\nabla k
  &= \nu_2 \diver\Big(\frac k\omega \,\nabla k\Big)+ \nu_0\frac
  k\omega \,\big|\bs{D}(\bs{u})\big|^2-\alpha_2 k\omega. 
\end{align}
\end{subequations}
Throughout the paper, bold letters denote functions with values in
$\mathbb{R}^3$ or $\mathbb{R}^9$ as well as normed spaces of such
functions. Here, the unknowns have the following physical
meaning: 
\begin{align*}
 \bs{u}& \ \text{ is the velocity of the mean flow,}\\
 p& \ \text{ is the average of the pressure,}\\
 \omega&\ \  \text{\parbox[t]{8.5cm}{is the average of the frequency associated
     with the turbulent kinetic energy,}}\\ 
 k& \ \text{ is the mean turbulent kinetic energy.} 
\end{align*}
The velocity field $\bs{v}$ of the fluid motion is given by
$\bs{v}=\bs{u}+\tilde{\bs{u}}$, where $\tilde{\bs{u}}$ denotes the
turbulent fluctuation velocity, such that the scalar $k$
is the time average  $\overline{\frac 12
\,|\tilde{\bs{u}}|^2}$. Further,
\begin{align*}
 & \nu_0,\ \nu_1,\ \nu_2 >0 \text{ and } \alpha_2,\ \alpha_1 >0 
    \text{ are dimensionless constant;} \\
 &\bs{f} \text{ is a given averaged external force,}\\
 &\bs{D}(\bs{u})=\frac12\big(\nabla\bs{u}+(\nabla\bs{u})^\top\big) 
 \text{ is the mean strain-rate tensor}.
\end{align*}
The function $\nu_0\frac k\omega$ denotes the kinematic eddy viscosity, while
$\nu_1 \frac k\omega$ and $\nu_2\frac k\omega$ denote the corresponding
diffusion constants for the scalars $\omega$ and $k$. The constants
$\nu_0,\ \nu_1,\ \nu_2 >0$ and $\alpha_2,\ \alpha_1 >0 $ in 
\eqref{eq:I.1-4} are related to the constants $A,\,A',\,A''$ \cite{Ref14}
(cf.\ also \cite[p.\.213]{Ref25} where $b=\frac23k$) as follows:
\begin{equation}
  \label{eq:KolmoValues}
  \nu_0=\frac43 A, \quad \nu_1=\frac23 A',\quad \nu_2 = \frac23 A'', \quad
\alpha_1=\frac7{11}, \quad \alpha_2=1.
\end{equation}
In Section \ref{s2} we discuss the scaling properties of the two-equation
model \eqref{eq:I.1-4} with the special viscosities ``$\nu_j\, k/\omega$'' and
loss terms ``$\alpha_1\omega^2$'' and ``$\alpha_2\,k\omega$''. These specific
choices of power-law nonlinearities relate to specific scaling laws in free
turbulence. In \cite{Ref14}, there is no indication why the particular
values of $\alpha_1$ and $\alpha_2$ were chosen.

Since the numerical values of $\nu_1$ and $\nu_2$ 
are not relevant for the existence theory of weak solutions for
\eqref{eq:I.1-4} we are going to develop below, we assume them
to be equal to 1. A detailed discussion of the numerical values of
closure coefficients and their role in turbulence modeling can be
found, e.g., in \cite{Ref2} and \cite[Chap.\,4.3.1]{Ref29}. 
However, we keep the coefficient $\nu_0$ to emphasize that the viscous
dissipation generated by the viscous term in \eqref{1.1} is feeding
into the mean turbulent kinetic energy, see the second last term in
\eqref{1.4}. Hence, for sufficiently smooth solutions we have the
formal energy relation
\begin{equation}
  \label{eq:I.Energy}
  \frac{\mathrm d}{\mathrm d t} \int_{\mathbb R^3} \Big(
\frac12|\bs{u}|^2 + k\Big) \dd x =  \int_{\mathbb R^3}\Big(\bs f \cdot
\bs u -\alpha_2 \omega k \Big) \dd x ,
\end{equation}
where the first term on the right-hand side gives the power of the
external forces, while the second term is Kolmogorov's way of modeling
dissipative losses, e.g.\ through thermal radiation.  We refer to
\cite{ObeBus02TT, ChaLew14MNFT} for general issues in turbulent
modeling, in particular to \cite[Ch.\,7+8]{ChaLew14MNFT} for the
mathematical analysis of the NS-TKE model (Navier-Stokes equation with
Turbulent Kinetic Energy), where the equation \eqref{1.3} for $\omega$
is absent and the energetic losses in \eqref{1.4} are modeled via
$k^{3/2}/\ell$ with a suitable mixing length $\ell$ 
instead of $\alpha_2k\omega$ (see
e.g.\ \cite[Eqn.\,(4.137)]{ChaLew14MNFT}.

System \eqref{eq:I.1-4} is an outgrowth of A.N. Kolmogorov's
theory of turbulence published in a series of papers in
1941. Comprehensive presentations of this theory can be found, e.g.,
in \cite{Ref10} and \cite[Vol.\,I, Chap.\,6.1,\,6.2; Vol.\,II,
Chap.\,8]{Ref21} (see also the article \cite[pp.\,488--503]{Ref27}). The
function $L=\frac{k^{1/2}}\omega$ (``external length scale'' or ``size
of largest eddies'') plays an important role for the study of the
energy spectrum of the turbulence (see \cite[Chap.\,33]{Ref15},
\cite[Chap.\,8.1]{Ref29}). A review of the work of A.N. Kolmogorov and the
Russian school of turbulence can be found in \cite{Ref30}. This paper
contains also some remarks about a possibly ``missing source term'' in
\eqref{1.3} (cf.\ \cite[p.\,212]{Ref25}).

A profound discussion of the mathematical background of
Obukhoff--Kolmogorov's spectral theory of turbulence (K41-functions,
bounds for the energy spectrum for low and high frequencies) is given
in \cite{Ref28}. 

In \cite{Ref6}, the authors study system \eqref{eq:I.1-4}  in
$\Omega\times\,]\,0,T\,[\,$, where $\Omega\subset\mathbb{R}^3$ is a
bounded $C^{1,1}$ domain, with mixed
boundary conditions for $\omega$ and $k$, the condition
$\bs{u}\cdot\bs{n}=0$ and a condition for the normal traction of the
tensor $-p\bs{I}+\nu_0\,\frac{k}{\omega}\,\bs{D}(\bs{u})$ on
$\partial\Omega\times\INTT$. Under these boundary
conditions, system \eqref{eq:I.1-4}  characterizes a
wall-bounded turbulent motion, i.e., turbulence is generated at the
Dirichlet part of the boundary. The authors complete this
boundary value problem by the initial conditions \eqref{1.6} and  
prove the existence of a weak solution by combining a truncation
method and the Galerkin approximation. Wall-generated turbulence is 
an important topic in engineering applications where two-equation models,
including the $k\,$-\,$\eps$ model, are heavily used, see
\cite{ChaLew14MNFT} and the references there.\smallskip 

The emphasis of this paper is quite different as we are interested in free
turbulence (also called isotropic or homogeneous turbulence) that develops
far away of the boundary and is rather governed by suitable scaling 
symmetries in the sense of \cite{Ober02SIST} and \cite{KlObPl20STM}. 
In \cite{{Ref14}} Kolmogorov writes about the derivation of his model:
``\emph{We may submit to a rather less complete mathematical investigation the
  turbulent motion which is homogeneous and isotropic (in all scales), and from
  which mean flow is absent; such a flow decays continuously with time. \ldots\
  Starting from the above local properties of turbulence (and with the help of
  some more coarsely approximate assumptions), we may construct the following
  complete system of equations to describe turbulent motion:}'' and then he
states his two-equation model (cited from English translation in \cite{Ref25}).

To preserve these similarity transforms we avoid boundaries and use periodic
boundary conditions and on a cube size with side length $\boxlength$, that can
be chosen much larger than the structures under consideration.  A bonus of the
scaling invariance of \eqref{eq:I.1-4} for $\bs f\equiv 0$ is the existence of
a rich class of similarity solutions.  Compatible with the periodic boundary
conditions we have the following explicit spatially constant solutions
\begin{equation}
  \label{eq:Solution}
  \bs u\equiv \bs u_\circ, \quad p\equiv 0, \quad
  \omega(t)=\frac{\omega_\circ}{ 1{+}\alpha_1\omega_\circ t}, 
  \quad 
  k(t)= \frac{k_\circ}{(1{+}\alpha_1 \omega_\circ t)^{\alpha_2/\alpha_1}},
\end{equation}
i.e.\ the mean turbulent kinetic energy decays like
$t^{-\alpha_2/\alpha_1}$, if there is no feeding through macroscopic
viscous dissipation. Indeed, independent of $\bs u$ and $k$, the equation
\eqref{1.3} for $\omega$ can always be solved by the spatially constant
solution $\omega(x,t)= \omega_\circ/(1{+}\alpha_1 \omega_\circ t)$. The
occurrence of asymptocially self-similar behavior for $\Omega = \mathbb R^d$ for a
closely related, but much simpler coupled system (obtained by replacing the
Navier-Stokes equation by a scalar equation for shear flows and neglecting
lower order terms) is discussed in \cite{Miel21?TCDP}. 

To show the effect of energy feeding from viscous dissipation
into the turbulent kinetic energy $k$ via the source term $\nu_0\frac
k\omega |\bs D(\bs u)|^2$ we can look at the following family of exact
shear flow solutions: 
\begin{align}
\label{eq:DecaySol}
\bs u(x,t)= \frac{U}{1{+}\alpha_1 \omega_\circ t} 
 \begin{pmatrix} \sin(\lambda x_3)\\ \cos(\lambda x_3)\\ 0 
   \end{pmatrix}, \quad 
 \omega(x,t)= \frac{\omega_\circ}{1{+}\alpha_1\omega_\circ t} , \quad 
 k(x,t)= \frac{k_\circ}{(1{+}\alpha_1\omega_\circ t)^2}. 
\end{align}
with $p\equiv 0$, where the positive constant parameters $\omega_\circ$,
$k_\circ$, $\lambda$, and $U$ are related by 
\[
U^2 = \frac{\alpha_2-2\alpha_1}{\alpha_1} \, k_\circ \quad \text{ and }
\quad \lambda^2 = \frac{2\alpha_1}{\nu_0}\, \frac{\omega_\circ^2}{k_\circ}\,.
\]
These solutions only exist for the case $\alpha_2/\alpha_1>2$, and
thus the decay of $k$ like $1/t^2$ is slower than
$1/t^{\alpha_2/\alpha_1}$ in \eqref{eq:Solution}, because of the
spatially constant 
source term $\nu_0\frac
k\omega |\bs D(\bs u)|^2 = \alpha_1 \omega_\circ U^2
(1{+}\alpha_1\omega_\circ t)^{-3}$. As in \cite{Ober02SIST} these invariant
solutions exist because of the scaling symmetries, and moreover they
are indeed compatible with period boundary conditions if $\lambda
\boxlength \in 2\pi \mathbb N$. For a given $\boxlength$ we find
infinitely many solutions by choosing $\lambda_n=2\pi n/\boxlength$ and
suitable $k_\circ$ and $\omega_\circ$. This also highlights the fact
that there are no uniform compactness properties unless we prescribe
a lower bound for $k$. 

In place of $\mathbb{R}^3\times \INTinf$, in the present
paper we study system \eqref{eq:I.1-4}  in the space-time
cylinder $Q=\Omega \times \INTT$, where
$\Omega=\big(\INTell\big)^3$ with $T,\boxlength>0$ arbitrary but
fixed. To implement periodic boundary 
conditions we interpret $\Omega$ as a torus by identifying the
opposite sides. If $\partial\Omega$ denotes the boundary of the cube
$\Omega \subset \mathbb{R}^3$ we set 
\[
 \Gamma_i=\partial\Omega\cap\{x_i=0\},\quad
 \Gamma_{i+3}=\partial\Omega\cap \{x_{i}=\boxlength \} \quad \text{
   for } i=1,2,3,  
\]
and complement \eqref{eq:I.1-4}  with periodic boundary
conditions and initial conditions as follows: 
\begin{subequations}
\label{eq:I.IBC}
\begin{align}
 \label{1.5}
& \left.\begin{array}{@{}l}
        \bs{u}\big|_{\Gamma_i\times\,]\,0,T\,[}=\bs{u}\big|_{\Gamma_{i+3}\times\,]\,0,T\,[\,},
        \qquad\qquad\text{ analogously for   }\; p,\omega,k,\\[2mm]
        \bs{D}(\bs{u})\big|_{\Gamma_i\times\,]\,0,T\,[\,}=\bs{D}(\bs{u})
        \big|_{\Gamma_{i+3}\times\,]\,0,T\,[\,},
        \ \text{ analogously for }\; \nabla\omega, \nabla k\\[2mm]
        \text{for }i=1,2,3;
       \end{array}\right\}\\[1.5mm]
\label{1.6}
&\;\bs{u}=\bs{u}_0,\;\; \omega=\omega_0, \;\; k=k_0\;\text{ in }\; \Omega\times\{0\}.
\end{align}
\end{subequations}
Initial/boundary-value problem \eqref{eq:I.1-4}  and 
\eqref{eq:I.IBC}  characterizes a turbulent motion of an
incompressible fluid in $Q$ that evolves from
$\{\bs{u}_0,\omega_0,k_0\}$ at time $t=0$. We assume the pressure
to be periodic thus avoiding additional pressure gradients that might occur
when assuming that $\nabla  p$ is periodic only. As a consequence the mean
flow $\boxlength^{-3} \int_\Omega \bs u(x,t) \,\dd x $ is constant, when assuming
$\bs f\equiv 0$, cf.\ \cite{ChoIoo94CTP, KagWah97ECCR}. The usage of
periodic boundary conditions is common in theoretical investigations of the
Navier-Stokes equations and modeling of free turbulence, see e.g.\ 
\cite{FMRT01NSET,LayLew03SSSS,Ref10,Lewa06VLES,LedLew07RANS,Ref28}. 

On physical grounds, the size $\boxlength$ of the underlying cube $\Omega$
should be greater than certain quantities of the turbulent motion. A
detailed discussion of this aspect is given in \cite[pp.\,25--26,\,%
424--435]{Ref8} (cf.\ also item 2$^\circ$ below). This is one of the main
reasons why we consider a cube $\Omega$ of side length $\boxlength$ and
periodic boundary conditions which provides an analysis that is
completely independent of $\boxlength$. In particular, we can choose $\boxlength$
much bigger than the ``external length scale''
$L(x,t) := k(x,t)^{1/2} /\omega(x,t)$.

Our proof of the existence of weak solutions of
\eqref{eq:I.1-4}  and \eqref{eq:I.IBC},  which has been
already sketched in \cite{Ref20}, is entirely independent of the discussion in
\cite{Ref6}. More specifically, the basic aspects of our paper are:
\begin{itemize}
\item[$1^\circ$] In Section \ref{s3} we introduce the notion of weak
  solution $\{\bs{u},\omega,k\}$ with defect measure $\mu$ for
  \eqref{eq:I.1-4}  and \eqref{eq:I.IBC}. This notion
  leads to a balance law for $\int_\Omega k(x,\cdot)\dd x$ and
  gives a connection between the energy equality for $\frac
  12\int_\Omega\big|\bs{u}(x,\cdot)\big|^2\dd x$ and the
  vanishing of $\mu$, cf.\ Proposition \ref{pr:3.2} which states that
  \eqref{eq:I.Energy} holds if $\mu = 0$.
 
 \item[$2^\circ$] In Section \ref{s4} we present our existence theorem
   for weak solutions $\{\bs{u},\omega,k\}$ with defect measure
   $\mu$. Based on comparison arguments 
   with the explicit solution in \eqref{eq:Solution} our
   solutions $\{\bs u,\omega,k\}$ satisfy, for a.a.\ $(x,t)\in \Omega \times \INTT$, 
\begin{equation}
 \label{eq:I.comparison}
 \frac{\omega^*}{1{+}\alpha_1\omega^* t} \geq \omega(x,t) 
  \geq \frac{\omega_*}{1{+}\alpha_1\omega_* t} 
 \quad \text{and} \quad 
  k(x,t) \geq \frac{k_*}{(1{+}\alpha_1\omega_* t)^{\alpha_2/\alpha_1}}
  \,, 
\end{equation}
if the initial conditions in \eqref{1.6} satisfy the corresponding
estimates at $t=0$. It is important to preserve these estimates even
through the necessary approximations, since that provide a lower bound for the
diffusion coefficients $k/w$ in the three evolution equations. 

\item[$3^\circ$] Moreover, the bounds in \eqref{eq:I.comparison} provide
  a physically relevant lower bound for 
Kolmogorov's external length scale $L=k^{1/2}/\omega$, namely 
  \[
    L(x,t)=\frac{k(x,t)^{1/2}}{\omega(x,t)} \ge c\,(1{+}t)^{1-\alpha_2/(2\alpha_1)}\quad 
     \text{ for all } t\in[0,T\,],
  \]
 where $\alpha_2$ and $\alpha_1$ are from  \eqref{1.3} and \eqref{1.4}, and 
 where $c=\text{const}>0$ neither depends on $\boxlength$ nor on $T$
 (cf.\ Corollary~\ref{c3.1} in  Section \ref{s4}). Using A.N. Kolmogorov's
 values from \eqref{eq:KolmoValues} we have $\alpha_2/\alpha_1= {11}/7$ and 
 $L$ {\it grows at least as } $t^{3/14}$, which compares well to $t^{2/7} $
 mentioned in \cite{Ref14}). 

\item[$4^\circ$] The proof of our existence theorem is given in
  Section \ref{s5}. It is based on the existence of an approximate
  solution $\{\bs{u}_\eps,\omega_\eps,k_\eps\}$
  (without defect measure) of \eqref{eq:I.1-4}  and
  \eqref{eq:I.IBC},  establishing 
  a-priori estimates independently of $\eps$ and then carrying
  out the limit passage $\eps\to 0$. The existence of the
  approximate solutions is obtained by applying an abstract existence
  results for evolutionary equations with pseudo-monotone operators
  from \cite[Thm.\,8.9]{Roub13NPDE}, see Appendix \ref{s.A} for the
  details. 

\item[$5^\circ$] Our approach is easily adaptable to more
  general domains with suitable boundary conditions, and to 
  the full-space $\mathbb R^d$ with general $d\in \mathbb N$. However,
  for notational convenience and physical relevance we restrict
  ourselves to $d=3$ and the spatially periodic case.

\item[$6^\circ$] In \cite{Lewa97MACT} a simplified one-equation model
    of turbulence is studied, where a defect measure appears as well
    (see the pages 397 and 416 there). Weak solutions for the full
    one-equation model were obtained in \cite{BuLeMa11}.   
\end{itemize}

The parallel work in \cite{Ref6} developed completely independently to the
present work, which had its origin in \cite{Ref20}. The former work is based on
an intricate Galerkin approximation with several regularization parameters and
is devoted to the case of bounded domains with nontrivial (even non-smooth)
boundary conditions that can trigger the generation of turbulence. For the
initial condition $k_0:= k(\cdot,0)$ we rely on the stronger assumption
$k_0(x)\geq k_*>0$ to obtain the very explicit lower bound for $k(x,t)$ in
\eqref{eq:I.comparison} that is independent of the domain size
$\boxlength$.  In \cite{Ref6} it is sufficient to assume the 
much weaker condition $\min\{0,\log k_0\} \in L^1(\Omega)$, but estimates are
given in terms of domain-dependent constants.   Moreover, \cite{Ref6} has a
\emph{stronger notion} of solution that additionally guarantees the validity of
a local balance equation for the total energy density
$E(x,t)=k(x,t)+\frac12|\bs u(x,t)|^2$, see Remark \ref{rm:TotalEnergy} and
relation \eqref{eq:BuMa.TotalEnergy} there.  

In subsequent work we will investigate similarity solutions that are
induced by the scaling laws discussed in Section \ref{s2}. The most
challenging question will be the derivation of suitable solution
concepts that allow the turbulent kinetic energy $k$ to vanish on parts of
the domain. This would allow us to study the predictions of  the
Kolmogorov model \eqref{eq:I.1-4} in which way turbulent regions invade
non-turbulent regions.

\section{Scaling laws and similarity}\label{s2}
\setcounter{equation}{0}

We consider the free turbulent motion of an incompressible fluid in
$\mathbb{R}^3\times\INTinf$ which is governed by the
following system of PDEs (note that $\bs f\equiv 0$):
\begin{subequations}
\label{eq:n2} 
\begin{align}
 \label{n2.1}
 \diver\bs{u}&=0,\\
 \label{n2.2}
 \frac{\partial\bs{u}}{\partial t}+(\bs{u}\cdot\nabla)\bs{u}
 &=\diver\big(d_1(\omega,k)\bs{D}(\bs{u})\big)-\nabla p,\\
 \label{n2.3}
 \frac{\partial\omega}{\partial t}+\bs{u}\cdot\nabla\omega
 &=\diver\big(d_2(\omega,k)\nabla\omega\big)-g_2(\omega,k)\omega,\\
 \label{n2.4}
 \frac{\partial k}{\partial t}+\bs{u}\cdot\nabla k
&=\diver\big(d_3(\omega,k)\nabla
k\big)+d_1(\omega,k)\big|\bs{D}(\bs{u})\big|^2-g_3(\omega,k)k, 
\end{align}
\end{subequations}
where $\bs{u}$, $p$, $\omega$ and $k$ are the unknowns, and 
\begin{align*}
 d_i: \big(\INTinf\big)^2&\longrightarrow\
 \INTinf\qquad(i=1,2,3),\\ 
 g_m: \big(\INTinf\big)^2&\longrightarrow \
 \INTinf\qquad(m=2,3) 
\end{align*}
are given coefficients. The coefficient $d_1(\omega,k)$ represents a
``generalized'' viscosity of the fluid. System
 \eqref{eq:n2}  obviously includes Kolmogorov's
two-equation model \eqref{eq:I.1-4}  with
\begin{align*}
 & {\displaystyle d_1(\omega,k)=\nu_0\frac k\omega,\quad
   d_2(\omega,k)= \nu_1\frac k\omega, \quad 
   d_3(\omega,k)=\nu_2 \frac k\omega,}&\\[2mm]
 &g_2(\omega,k)=\alpha_1\omega,\quad g_3(\omega,k)=\alpha_2 \omega.&
\end{align*}
We want to show that these choices are special, because they give a
richer structure of scaling invariances than arbitrary nonlinear
functions. In particular, they respect the classical Reynolds symmetry
(see \cite[Sec.\,3.3]{ChaLew14MNFT}), but go one step beyond because
the viscosities $d_j(\omega,k)$ also have scaling properties. We refer
to \cite{Bare93SLFD, Ober02DEIT, Ober02SIST} where the importance of scaling
symmetries for the modeling of free turbulence is discussed.

Let $\{\bs{u},\omega,k\}$ be a classical solution of \eqref{eq:n2}
that has a suitable decay for $|x|\to \infty$ such that the following
integrals over $\mathbb R^3 $ exist. We multiply \eqref{n2.2} by
$\bs{u}$, integrate by parts over $\mathbb R^3 $, integrate
\eqref{n2.4} over $\mathbb R^3 $, and add the equations obtained. This
gives the energy balance
\begin{equation}
 \label{n2.5}
 \frac{\mathrm d}{\mathrm d t}\int_{\mathbb{R}^3} 
\Big(\frac 12|\bs{u}|^2+k\Big)\dd x=
 -\int_{\mathbb{R}^3}g_3(\omega,k)k\dd x,\quad t\in\INTinf,
\end{equation}
cf.\  Proposition \ref{pr:3.2} in Section \ref{s4}.
\medskip

We are now studying the invariance of $\{\bs{u},\omega,k\}$ under the
scaling 
\begin{equation}
 \label{n2.6}
 \partial_t\mapsto\alpha\partial_t,\quad \partial_{x_j}\mapsto
 \beta\partial_{x_j},\quad\bs{u}\mapsto
 \gamma\bs{u},\quad\omega\mapsto\rho\omega,\quad k\mapsto\sigma k, 
\end{equation}
where
$(\alpha,\beta,\gamma,\rho,\sigma)\in\big(\,]\,0,+\infty\,[\,\big)^5$. Here,
the pressure $p$ is omitted, for it can be always suitably scaled. In
addition to the well-known scaling laws for the Navier-Stokes
equations, the scaling \eqref{n2.6} have to leave invariant the
coefficients $d_i(\omega,k)$ and $g_m(\omega,t)$ for $i=1,2,3$ and $m=2,3$,
too.

To this end, we consider the following conditions for the family of parameters $(\alpha,\beta,\gamma,\rho,\sigma)$ and the coefficients $d_i$ and $g_m$:
\begin{align}
\label{n2.7}
\alpha=\beta\gamma,\qquad \sigma=\gamma^2,\qquad\qquad\qquad\qquad\\[3mm]
 \label{n2.8}
 \forall \; \omega,k >0:\ \left\{ 
   \begin{array}{l}    
        \beta^2d_i(\rho\omega,\sigma k)=\alpha d_i(\omega,k),\quad
        i=1,2,3,
   \\[1mm]
        g_m(\rho\omega,\sigma k)=\alpha g_m(\omega,k),\quad m=2,3.
       \end{array}\right.
\end{align}
The first condition in \eqref{n2.7} implies the invariance of the
convective derivative $\partial_t+\bs{u}\cdot\nabla$ under
\eqref{n2.6}, while the second condition implies that $|\bs{u}|^2$ and
$k$ have the same scaling property which is necessary for the
conservation law \eqref{n2.5} to hold. It is now easy to see that
system  \eqref{eq:n2}  is invariant under the scaling laws
\eqref{n2.6} if the conditions \eqref{n2.7} and \eqref{n2.8} hold.

In order to relate the present discussion to Kolmogorov's two-equation
model \eqref{eq:I.1-4}  we make an ``ansatz'' for the parameter
$\beta$ as well as for the coefficients $d_i$ and $g_m$. For
$(\gamma,\rho),(\omega,k)\in\big(\INTinf\big)^2$ define
\begin{equation}
 \label{n2.9}
 \beta=\rho^A\gamma^{1-2B}
\end{equation}
\begin{equation}
 \label{n2.10}
 d_i(\omega,k)=D_i\omega^{-A}k^B,\quad g_m(\omega,k)=G_m\omega^A k^{1-B},
\end{equation}
where $D_i$, $G_m$ ($i=1,2,3$; $m=2,3$) and $A$, $B$ are arbitrary
positive constants. Condition \eqref{n2.9} is equivalent to
\[
 \frac\beta\gamma\,\rho^{-A}\gamma^{2B}=1\quad\text{ resp. }\quad \frac1{\beta\gamma}\,\rho^A\gamma^{2(1-B)}=1.
\]
Observing \eqref{n2.7}, it is readily seen that $d_i$ and $g_m$ as in
\eqref{n2.10} obey the scaling conditions \eqref{n2.8} for all choices
of $D_i$, $G_m$, $A$, and $B$.

Finally, let $A=B=1$ in \eqref{n2.9} and \eqref{n2.10}, i.e.\ $g_m$ does
not depend on $k$. Then  we obtain
\[
 d_i(\omega,k)=D_i\frac k\omega,\quad g_m(\omega,k)=G_m\omega\quad (i=1,2,3; \; m=2,3).
\]
Hence, Kolmogorov's two-equation model of turbulence, which is
obtained for $D_i=\nu_{i-1}$, $G_2=\alpha_1$, and $G_3=\alpha_2$, is
invariant under the scaling \eqref{n2.6} with the two-parameter
family 
\begin{equation}
  \label{eq:2.11}
   (\rho,\gamma)\mapsto (\alpha,\beta,\gamma,\rho,\sigma)
   =\Big(\rho,\frac \rho\gamma,\gamma,\rho,\gamma^2\Big).
\end{equation}

\section{Definition of weak solutions}\label{s3}

We begin with introducing notations that will be used throughout the
paper.

Let $X$ denote any real normed space with norm $|\cdot|_X$, and let
$\langle x^*,x\rangle_X$ denote the dual pairing of $x^*\in X^*$ and $x\in
X$. By $L^p(0,T;X)$ ($1\le p\le+\infty$) we denote the vector space of all
equivalence classes of Bochner measurable mappings $u:[\,0,T\,]\to X$ such that
\[
  \|u\|_{L^p(0,T;X)} = \left\{ 
   \begin{array}{l@{\quad\text{if }\;}l}
     {\Big(\int_0^T\big|u(t)\big|_X^p\dd t\Big)^{1/p}}& 1\le p<+\infty,
     \\
     \mathop{\mathrm{ess\,sup}}\limits_{t\in[0,T\,]}\big|u(t)\big|_X&p=+\infty
   \end{array}
  \right.
\]
is finite (see e.g.\ \cite[Chap.\,III, \S3, Chap.\,IV, \S3]{Ref4},
\cite[App.]{Ref5} and \cite{Ref9} for details).
Let $\Omega\subseteq\mathbb{R}^N$ ($N\ge 2$) be any open set, and let
$Q=\Omega\times\INTT$ for $T>0$. For $1\le p<\infty$
and $u\in L^p(Q)$ define
\[
 [u](t)(\cdot)=u(\cdot,t)\quad\text{for a.a. }\: t\in[\,0,T\,].
\]
By Fubini's theorem, the function
$t\mapsto\int_\Omega\big|u(x,t)\big|^p\mathrm{d}x$ is in $L^1(0,T)$
and there holds 
\[
 \int_0^T\big\|[u](t)\big\|_{L^p(\Omega)}^p\dd t=\int_Q 
   \big|u(x,t)\big|^p\dd x\dd t.
\]
An elementary argument shows that the mapping $u\mapsto[u]$ is a
linear isometry of $L^p(Q)$ onto
$L^p\big(0,T;L^p(\Omega)\big)$. Therefore, these spaces will be
identified in what follows. By $W^{1,p}(\Omega)$ we denote the usual
Sobolev space, and we set $\bs{W}^{1,p}(\Omega) =
\big(W^{1.p}(\Omega)\big)^N$.  \smallskip

Unless otherwise stated, from now on let $\Omega=\big(\INTell\big)^3$
denote the cube introduced in Section~\ref{s1}. We define 
\begin{align*}
  & W_{\mathrm{per}}^{1,p}(\Omega) =\big\{u\in
  W^{1,p}(\Omega);\:u\big|_{\Gamma_i}=u\big|_{\Gamma_{i+3}}\text{ for }i=1,2,3 \big\},
\\[1.5mm]
  &
  \bs{W}_{\mathrm{per},\diver}^{1,p}(\Omega)=\big\{\bs{u}\in
  \bs{W}_{\mathrm{per}}^{1,p}(\Omega);\,\diver\bs{u}=0\text{ a.e.\ in } 
  \Omega\big\},
\\[1.5mm]
  &C_{\mathrm{per},T}^1(\overline{Q})=\big\{\varphi\in
  C^1(\overline{Q});\ \varphi\big|_{\Gamma_i\times\INTT} = \varphi 
   \big|_{\Gamma_{i+3}\times\INTT} , \ 
   \nabla\varphi\big|_{\Gamma_i\times\INTT}
    = \nabla \varphi  \big|_{\Gamma_{i+3}\times\INTT} 
  \\
  &\hspace*{15em} \text{for }i=1,2,3,\ \varphi(x,T)=0\;\:\forall\: x\in\Omega\big\},\\[1.5mm]
  &\bs{C}_{\mathrm{per},T,\diver}^1(\overline{Q})=\big\{\bs{v}\in 
  \bs{C}_{\mathrm{per},T}^1(\overline{Q});\:\diver\bs{v}=0\text{ in }Q\big\}.
\end{align*}
We emphasize that the test functions in
$C^1_{\mathrm{per},T}(\overline{Q})$  vanish at $t=T$. 
Finally, by $\mathcal{M}_\geq(\overline{Q})$ we denote the set of all
non-negative, bounded Radon measures on the $\sigma$-algebra of Borel
sets $\subseteq\overline{Q}$, which is a closed cone in the vector space
$\mathcal{M}(\overline Q)\simeq C(\overline Q)^*$ of all (signed)
Radon measures. 

To simplify the notation we subsequently set $\alpha_1=1$ and
$\nu_2=1$, which can always be achieved by exploiting the scaling
\eqref{eq:2.11}. We further set $\nu_1=1$, but
keep the constant $\nu_0>0$ to emphasize that the source term in the
equation \eqref{1.4} for the turbulent energy $k$ arises from the
dissipation in the momentum equation \eqref{1.2} for $\bs u$.

\begin{definition} \label{def:WeakSol}
  Let $\bs{f} \in\bs{L}^1(Q)$, $\bs{u}_0\in \bs{L}^1(\Omega)$ and
  $\omega_0,k_0\in L^1(\Omega)$ such that $\omega_0,k_0\ge 0$ a.e.\ in
  $\Omega$. A triple of measurable functions $\{\bs{u},\omega,k\}$ in
  $Q$ is called \emph{weak solution of \eqref{eq:I.1-4}  and
    \eqref{eq:I.IBC}  with a non-negative defect measure}
  $\mu\in \mathcal{M}_\geq(\overline{Q})$, if
 \begin{equation}
  \label{2.1}
  \omega>0,\;\; \frac k\omega\ge \mathrm{const}>0\;\text{ a.e. in } Q,
 \end{equation}
\begin{equation}
 \label{2.2}
 \left.\begin{array}{l}
        \bs{u}\in L^\infty \big(0,T ;\bs{L}^2(\Omega)\big)\cap
        L^2\big(0,T;\bs{W}_{\mathrm{per},\diver}^{1,2}(\Omega)\big),
\\[2.5mm]
        \omega\in L^\infty \big(0,T ; L^2(\Omega)\big)\cap
         L^2\big(0,T;W_{\mathrm{per}}^{1,2}(\Omega)\big),
\\[2.5mm]
        k\in  L^\infty (0,T ; L^1(\Omega)\big)\cap
        L^{15/14}\big(0,T;W_{\mathrm{per}}^{1,15/14}(\Omega)\big), 
       \end{array}
\right\}
\end{equation}
\begin{equation}
\label{2.3}
\int_Q\frac k\omega\big(\big(1+\big|\bs{D}(\bs{u})\big|\big)\big|\bs{D}(\bs{u})\big|+|\nabla\omega|+|\nabla k|\big)\dd x\dd t<\infty,
\end{equation}
the following weak equations hold
\begin{equation}
 \label{2.4}
 \left.\begin{array}{l}
       {\displaystyle
         -\!\int_Q\!\bs{u}\cdot\frac{\partial\bs{v}}{\partial t}\dd
         x\dd t-\!\!\int_Q(\bs{u}\otimes\bs{u}):\nabla\bs{v}\,\dd x\dd
         t+\nu_0\int_Q\frac k\omega\bs{D}(\bs{u}):\bs{D}(\bs{v})\dd
         x\dd t}\\[0.9em] 
       =\displaystyle
         \int_\Omega\bs{u}_0(x)\cdot\bs{v}(x,0)\dd
         x+\int_Q\bs{f}\cdot\bs{v}\dd x\dd t\;\;\text{ for all
         }\bs{v}\in \bs{C}_{\mathrm{per},T,\diver}^1(\overline{Q}), 
       \end{array}\right\} 
\end{equation}
\begin{equation}
 \label{2.5}
 \left.\begin{array}{l}
        {\displaystyle -\int_Q\omega\frac{\partial\varphi}{\partial
            t}\,\dd x\dd t-\int_Q\omega\bs{u}\cdot\nabla\varphi\,\dd
          x\dd t+\int_Q\frac
          k\omega\,\nabla\omega\cdot\nabla\varphi\dd x\dd t}\\[0.9em] 
        {\displaystyle=\int_\Omega\omega_0(x)\varphi(x,0)
        \dd x-\int_Q\omega^2\varphi\,\dd x\dd t\;\;
  \text{ for all }\varphi\in C_{\mathrm{per},T}^1(\overline{Q}),}
       \end{array}\right\}
\end{equation}
\begin{equation}
 \label{2.6}
 \left.\begin{array}{l}
        {\displaystyle-\int_Q k\,\frac{\partial z}{\partial t}\,\dd x\dd t-\int_Q k\bs{u}\cdot\nabla z\,\dd x\dd t+\int_Q\frac k\omega\,\nabla k\cdot\nabla z\,\dd x\dd t}\\[0.9em]
        {\displaystyle=\int_\Omega k_0(x)z(x,0)\dd x+\int_Q\Big(\nu_0\frac k\omega\,\big|\bs{D}(\bs{u})\big|^2-\alpha_2 k\omega\Big)z\,\dd x\dd t}\\[0.9em]
        {\displaystyle\quad+\int_{\overline{Q}} z\,\dd
          \mu\quad\text{ for all } z\in C_{\mathrm{per},T}^1(\overline{Q}),}
       \end{array}\right\}
\end{equation}
the Leray-Hopf type energy bound for the Navier-Stokes equation
\begin{equation}
  \label{eq:NS.Ener}
  \left.\begin{array}{l}
        {\displaystyle \int_\Omega  \frac12\big|\bs{u}(x,t)\big|^2 \dd
          x+\int_0^t \!\!\int_\Omega \nu_0\frac k\omega \big| \bs{D}(\bs
          u)\big|^2 \,\dd x\dd s } 
\\[0.9em]
       \leq  {\displaystyle \int_\Omega 
     \frac12\big|\bs{u}_0(x)\big|^2 \dd x + 
   \int_0^t\!\!\int_\Omega\bs{f}\cdot\bs{u}\,\dd x\dd s}
       \end{array}\right\}\text{ for a.a. } t\in[0,T],
\end{equation}
and the total energy satisfies the estimate
\begin{equation}
 \label{eq:EstimTotEner}
\left.\begin{array}{l}
        {\displaystyle \int_\Omega  \! \Big(
          \frac12\big|\bs{u}(x,t)\big|^2 {+}k(x,t) \Big) \dd
          x+\int_0^t\!\!\int_\Omega \alpha_2k\omega\,\dd x\dd s } 
\\[0.9em]
        {\displaystyle\le\int_\Omega \! \Big( \frac12 
   \big|\bs{u}_0(x)\big|^2{+} k_0(x)\Big)\dd x + 
   \int_0^t\!\!\int_\Omega\bs{f}\cdot\bs{u}\,\dd x\dd s}
       \end{array}\right\}\text{ for a.a. } t\in[0,T].
\end{equation}
\end{definition}

\noindent
It is easy to see that all integrals in \eqref{2.4}--\eqref{2.6} are
well-defined. It suffices to consider the integrals with integrands
$k\bs{u}\cdot\nabla z$ and $\frac k\omega\big|\bs{D}(\bs{u})\big|^2z$ in
\eqref{2.6}. Firstly, it is well-known that condition \eqref{2.2} on $\bs{u}$
implies $\bs{u}\in \bs{L}^{10/3}(Q)$ (combine H\"older's inequality and
Sobolev's embedding theorem). Analogously, the condition \eqref{2.2} on $k$
implies $k\in L^{10/7}(Q)$ (take $N=3$, $\theta =3/4$,
$(p_1,p_2)=(1,\frac{15}{14})$, and $(s_1,s_2)=(\infty,\frac{15}{14})$ in
Lemma~\ref{lem:GN}(B) below). Hence, $k\bs{u}\in \bs{L}^1(Q)$. Secondly,
$\frac k\omega\,\big|\bs{D}(\bs{u})\big|^2\in L^1(Q)$ by virtue of \eqref{2.3}.

\begin{remark}\label{r2.neu}\slshape
The condition $k/\omega \geq \text{const} >0$ is crucial for our
existence theory, in particular for obtaining the regularities for $\{\bs u,
\omega, k\}$ stated in \eqref{2.2}. It would be desirable to develop
an existence theory without this condition, because this would allow
us to study  how the support of $k$, which may be called the
`turbulent region', invades the `non-turbulent region' where $k\equiv 0$. 
\end{remark} 

\begin{remark}[Classical solutions are weak solutions]
\label{re:ClassWeakSol}\slshape
Every sufficiently regular classical solution $\{\bs{u},\omega,k\}$ of
\eqref{eq:I.1-4} and \eqref{eq:I.IBC} satisfies the variational identities
\eqref{2.4}, \eqref{2.5} and \eqref{2.6} with defect measure $\mu=0$. To verify
this, we multiply \eqref{1.2}, \eqref{1.3} and \eqref{1.4} by the test
functions $\bs{v}$, $\varphi$ and $z$, respectively, and integrate by parts
over the cube $\Omega$ and then over the interval $[\,0,T\,]$.
Moreover, it is easy to see that the energy inequalities
\eqref{eq:NS.Ener} and \eqref{eq:EstimTotEner} hold as equalities. 
\end{remark}

Of course, the important implication to be shown is that smooth weak
solutions are indeed classical solutions. In order to establish this, we
crucially use that the inequality \eqref{eq:EstimTotEner} for the total energy 
$\int_\Omega \big(\frac12|\bs u|^2 + k\big) \dd x$ and combine it with the upper
estimate \eqref{eq:NS.Ener} for the macroscopic kinetic 
energy $\int_\Omega \frac12|\bs u|^2  \dd x$ and a lower energy estimate for
the turbulent kinetic energy $\int_\Omega  k \,\dd x$, which
will be derived next. 

\begin{lemma} 
  Let $\{\bs{u},\omega,k\}$ be a weak solution of \eqref{eq:I.1-4} and
  \eqref{eq:I.IBC} with defect measure $\mu$. Then, we have the integral
  relations
\begin{subequations}
\begin{equation}
 \label{2.8}
 \int_\Omega\omega(x,t)\dd
 x+\int_0^t\int_\Omega\omega^2\dd x\dd
 s=\int_\Omega\omega_0(x)\dd x\;\;\text{ for all }t\in[0,T\,],
\hspace*{2em}\vspace{0.5em}
\end{equation}
\begin{equation}
 \label{2.9}
 \left.\begin{array}{l}
        {\displaystyle\int_\Omega k(x,t)\dd x=\int_\Omega k_0(x)\dd x
        +\int_0^t\!\!\int_\Omega\Big(\nu_0\frac
        k\omega\,\big|\bs{D}(\bs{u})\big|^2-\alpha_2 k\omega\Big)\dd
        x\dd s}
\\[0.9em]
        \hspace*{6em}{}+{}\mu\big(\overline{\Omega}{\times}[0,t\,]\big)\;
        \text{ for a.a. } t\in [0,T\,],
       \end{array}
\right\}
\vspace{0.5em}
\end{equation}
\begin{equation}
 \label{2.10}
 \lim\limits_{t\to 0}\int_\Omega k(x,t)\dd x=\int_\Omega k_0(x)\dd x 
                               +\mu\big(\overline{\Omega} {\times}\{0\}\big),
\hspace*{10em}\vspace{0.5em}
\end{equation}
\begin{equation}
 \label{2.11}
 \left.\begin{array}{l}
        {\displaystyle\int_\Omega k(x,t)\dd x=\int_\Omega k(x,s)\dd
          x+\int_s^t\!\!\int_\Omega\Big(\nu_0\frac
          k\omega\,\big|\bs{D}(\bs{u})\big|^2-\alpha_2 k\omega\Big)\,\dd x\dd
          \tau}\\[0.9em] 
   \ \qquad\qquad\qquad+\mu\big(\overline{\Omega}\times\,]\,s,t\,]\big)\; 
   \text{ for a.a.\ } s,t\in[0,T\,] \text{ with } s<t.
       \end{array}\right\}
\end{equation}
\end{subequations}
\end{lemma}
\begin{proof} 
It suffices to prove \eqref{2.9}. The same reasoning gives \eqref{2.8}, and
the relations \eqref{2.10} and \eqref{2.11} follow from \eqref{2.9}.  For
$t\in \:]\,0,T\,[\,$ and $m>\frac1{T-t}$ with $m\in \mathbb{N}$ we define
 \[
  \eta_m(s)=\left\{\begin{array}{c@{\quad\text{if }\;\;}l}
                       1 &0\le s \le t,\\[0.1em]
                       m(t{-}s)+1& t \leq s  \leq t+\frac1m,\\[0.1em]
                       0& s\geq t+\frac1m.
                      \end{array}\right.
 \]
Then, $\eta_m\in C({[0,\infty[})$ and $\dot \eta_m = m
1\!\!1_{{]t,t+1/m[}}$. For the Steklov average 
\[
\eta_{m,\lambda}(s) = \frac1\lambda \int_s^{s+\lambda} \eta_m(\tau) \dd \tau,
\quad s\geq 0,\ \lambda>0,
\]
we find 
\[
\eta_{m,\lambda} \in C^1([0,T]), \quad \eta_{m,\lambda}\overset{\lambda \to
  0^+}{\longrightarrow} \eta_m \text{ in } C^0([0,T]), \quad
\eta_{m,\lambda}(0)=1 \text{ and } \eta_{m,\lambda}(T)=0 
\]
for $\lambda \in {]0,t[}$. Moreover, we have 
$\dot \eta_{m,\lambda}(s) \to \dot \eta_m(s) \text{ for all } s \in
[0,T]\setminus\{t, t{+}\tfrac1m\}$, 
and for all $f \in L^1(Q_T)$ we find
\[
\int_{Q_T} f(x,s) \dot \eta_{m,\lambda}(s) \dd x \dd s \longrightarrow -m
\int_{t}^{t+1/m} \int_\Omega f(x,s)\dd x \dd s \quad \text{as } \lambda\to 0^+.
\]

Inserting the function $z(x,s)=1\!\!1_\Omega(x)\eta_{m,\lambda}(s)$ for
$ (x,s) \in Q$ into \eqref{2.6} with $s$ in place of $t$, the limit $\lambda\to
0^+$ leads to the relation 
\begin{align}
m\! \int_t^{t+1/m}\!\!\!\int_\Omega k(x,s)\,\dd x\dd s
\nonumber
 &=\int_\Omega k_0(x)\dd x+\int_0^{t+1/m}\!\!\!\int_\Omega \!
\Big(\nu_0\frac k\omega\,\big|\bs{D}(\bs{u})\big|^2- 
  \alpha_2 k\omega\Big)\eta_m\,\dd x\dd\tau\\
   \label{2.12}
 &\quad+\mu\big(\overline{\Omega}{\times}[\,0,t\,]\big)
 +\int_{\overline{\Omega}\times]t,t+\frac1m[}\!\!\eta_m(\tau)\dd\mu.
\end{align}
Because of $\eta_m(\tau)\in [0,1]$ for all $s \in [0,T]$ we have 
$ \int_{\overline{\Omega}\times]t,t+\frac1m[} \eta_m(\tau) \,\dd\mu \leq
\mu\big(\overline{\Omega}\times {]t,t+\frac1m[}\big) \rightarrow 0$ as
$m\to\infty$. The limit passage $m\to \infty$ in \eqref{2.12} gives
\eqref{2.9} for every Lebesgue point $t\in [0,T\,]$ of the function
$t\mapsto\int_\Omega k(x,t)\dd x$.
\end{proof}

We are now ready to show that smooth enough weak solutions are indeed
classical solutions and that the associated defect measure has to vanish. . 

\begin{proposition}
[Smooth weak solutions are classical]
\label{pr:SmoothWeakSol}
If $\{\bs{u},\omega,k\}$ is a weak solution of \eqref{eq:I.1-4} and
\eqref{eq:I.IBC} with defect measure $\mu$ (in the
sense of Definition \ref{def:WeakSol}) such that $\bs{u}$, $\omega$, and $k$
are sufficiently smooth (e.g.\ twice 
continuously differentiable in $x$ and once in $t$), then 
$\{\bs{u},\omega,k\}$ is a classical solution of  \eqref{eq:I.1-4} and
\eqref{eq:I.IBC}. 
\end{proposition}
\begin{proof} By definition weak solutions lie in $\bs W^{1,2}_\text{per,div}
(\Omega)$, which implies \eqref{1.1}. Similarly, the periodic boundary 
conditions \eqref{1.5} follow from the choice of spaces for the weak solution.   

Using the smoothness of $\{\bs{u},\omega,k\}$ we can integrate by parts in the
weak equations \eqref{2.4} and \eqref{2.5}. From this we obtain the validity of
the classical equations \eqref{1.2} and \eqref{1.3} for $\bs u$ and $\omega$,
respectively, and the initial conditions $\bs u(0,\cdot)=\bs u_0$ and
$\omega(0,\cdot)=\omega_0$.
    
Since the Navier-Stokes equation is classically satisfied, the kinetic energy
satisfies  \eqref{eq:NS.Ener} with equality. Adding this equality to relation
\eqref{2.9} for the turbulent energy, the term $\nu_0 \frac k\omega|\bs D(\bs
u)|^2$ exactly cancels; and we obtain  
\[
\int_\Omega \!\!\big( \frac12 |\bs u(x,t)|^2{+} k(x,t)\big) \dd x = \!\int_\Omega 
\!\!\big( \frac12|\bs u_0|^2 {+} k_0\big) \dd x + \! \int_0^t\!\!\int_\Omega \!\! 
\big( \bs f{\cdot} \bs u {-}\alpha_2 k\omega \big) \dd x \dd s 
+ \mu(\overline{\Omega}{\times} [0,t])
\]
for a.a.\ $t\in [0,T]$.  Comparing this to the total energy inequality
\eqref{eq:EstimTotEner} and using $\mu\geq 0$, we conclude
$ \mu(\overline{\Omega}{\times} [0,t]) = 0 $ for a.a.\ $t\in [0,T]$. Thus, we find
$\mu(\overline\Omega {\times} {[0,T[})=0$ which gives $\int_{\overline Q} z \dd \mu=
\int_{\overline\Omega}z(x,T) \dd\mu(T,x) =0$ in
\eqref{2.6}. For the last identity we exploit that $z\in
C^1_{\text{per},T}(\overline Q) $ implies $z(x,T)=0$ on $\overline\Omega$. 

Again, using the smoothness of $\{\bs{u},\omega,k\}$ we can integrate by
parts in the weak equations \eqref{2.6} and obtain
the validity of the classical equations \eqref{1.4}  and the initial conditions
$k(0,\cdot)=k_0$. 
\end{proof}

We note that by \eqref{2.9} the defect measure $\mu\geq 0$ contributes
positively to the integrated turbulent energy $\int_\Omega k(x,t) \dd
x$. In contrast, the energy inequality \eqref{eq:NS.Ener} for weak
solutions of the Navier-Stokes equations provides an upper bound for the
integrated kinetic energy $\int_\Omega \frac12 |\bs{u}(x,t)|^2 \dd x$ in terms of possibly
different defect measure $\mu_\mathrm{NS}$.  The expectation is that these
two measures exactly cancel each other when considering the total kinetic
energy $\int_\Omega \big( \frac12 |\bs{u}(x,t)|^2 +k(x,t)\big) \dd x $, 
and then \eqref{eq:EstimTotEner} holds as an equality. Our 
methods will not be strong enough to show this cancellation but we establish the
corresponding upper bound stated in  \eqref{eq:EstimTotEner}, which may be
interpreted as $\mu \leq \mu_\mathrm{NS}$.   In the
related work \cite{Ref6} the desired cancellation is derived by completely
different methods.

\begin{remark}
[Conservation law for the energy density $E$] 
\label{rm:TotalEnergy}\slshape 
For fluid models involving an additional energy
  equation, it is natural to derive equations for the total energy density,
  which in our case reads $E(x,t)=k(x,t)+\frac12 |\bs u(x,t)|^2$. This idea
  goes back to Feireisl and M\'alek in \cite{FeiMal06,BuFeMa09} and provides a
  local balance law for the total energy density $E$.  We expect that the
  result of \cite[Thm.\,1.1, Eqn.\,(1.50)]{Ref6} also holds in our case and 
  conjecture that there exist weak solutions as stated in Theorem
  \ref{th:MainExist} that 
  additionally satisfy the distributional form of the local balance equation
\begin{equation}
  \label{eq:BuMa.TotalEnergy}
  \frac{\partial}{\partial t} E + \diver\big( (E{+}p)\bs u\big) = 
\diver\Big(\frac k\omega \nabla k + \nu_0 \frac k\omega \bs D(\bs u) \, \bs u
\Big)  + \bs f  {\cdot}\;\! \bs  u - \alpha_2 k \omega,
\end{equation}
A close inspection of our estimates shows that all
terms in this equation can be defined as distributions, if the pressure $p$ is
recovered from \eqref{1.2} in the standard way. However, at present it remains
unclear how this relation can be derived using our approach based on
pseudo-monotone operators.
\end{remark}

Clearly, integrating the local balance law \eqref{eq:BuMa.TotalEnergy}
over $\Omega$ and using the periodic boundary condition implies that the
total-energy inequality \eqref{eq:EstimTotEner} holds as equality:
\begin{equation}
  \label{eq:IntegEnerg}
  \begin{aligned}
 &\int_\Omega\Big(\frac12\big|\bs{u}(x,t)\big|^2+k(x,t)\Big)\dd
 x+\alpha_2\int_0^t\!\!\int_\Omega k\omega\dd x\dd s\\ 
 & \quad =\int_\Omega\Big(\frac12\big|\bs{u}_0(x)\big|^2+k_0(x)\Big)\dd
 x+\int_0^t\!\!\int_\Omega\bs{f}\cdot\bs{u}\dd x\dd s \quad \text{for all }
 t\in [0,T].
\end{aligned}
\end{equation}
The following result shows that in this case the defect measure $\mu$ in
\eqref{2.6} is closely related to the defect measure associated with the weak
solution of the Navier-Stokes equation. The result follows simply by
subtracting \eqref{2.9} from \eqref{eq:IntegEnerg}.

\begin{proposition}[Energy equalities and defect measure] \label{pr:3.2} 
Let $\{\bs{u},\omega,k\}$ and $\mu$ be a weak solution as in Definition
\ref{def:WeakSol}. If additionally the energy equality \eqref{eq:IntegEnerg}
holds, then the following two statements are equivalent:
\smallskip

(i) \ \quad $\mu=0$;
\smallskip

(ii) \quad ${\displaystyle \int_\Omega\frac12\big|\bs{u}(x,t)\big|^2\dd x 
 +\nu_0\int_0^t\!\!\int_\Omega\frac k\omega 
  \big|\bs{D}(\bs{u})\big|^2\dd x\dd s}$
\\[0.3em]
\hspace*{1.2cm} ${\displaystyle
  \qquad =\frac12\int_\Omega\big|\bs{u}_0(x)\big|^2\dd
  x+\int_0^t\!\!\int_\Omega\bs{f}\cdot\bs{u}\dd x \dd
  s\;\text{ for a.a. }\: t\in[\,0,T\,]}$.
\end{proposition}

This result shows that the two energy inequalities \eqref{eq:NS.Ener},
\eqref{eq:EstimTotEner} and the 
defect measure $\mu$ in \eqref{2.6} are related to the classical problem of
proving an energy equality for weak solutions of the Navier-Stokes 
equations. A similar result for the case of Navier-Stokes equations
with temperature dependent viscosities has been obtained in
\cite{Ref22}. 
Defect measures also appear in a natural way in the context of weak
solutions of other types of nonlinear PDEs (see e.g.\
\cite{Ref1,Ref11,Ref16}).

\section{An existence theorem for weak solutions}\label{s4}
\setcounter{equation}{0}

We define the function spaces 
\begin{align*}
 & C_{\mathrm{per}}^\infty(\Omega)=\big\{u|_\Omega\; ; \ u\in C^\infty(\mathbb{R}^3), u \,\text{ is $\boxlength$-periodic}\\
 &\hspace*{4cm} \text{ in the directions }\:\bs{e}_1, \bs{e}_2,\bs{e}_3\big\},\\[1.5mm]
 &\bs{C}_{\mathrm{per},\diver}^\infty(\Omega)=\big\{\bs{u}\in
 \bs{C}_{\mathrm{per}}^\infty(\Omega)\; ; \ \diver\bs{u}=0\;\text{ in }\: \Omega\big\}.
\end{align*}
We impose the following conditions upon the right-hand side in
\eqref{1.2} and the initial data in \eqref{1.6}:  
\begin{equation}
 \label{3.1}
 \left.\begin{array}{l}
   \bs{f}\in \bs{L}^2(Q);\  \bs{u}_0 \in
   \bs{L}^2_{\diver}(\Omega):= \overline{\bs{C}_{\mathrm{per},\diver}^\infty(\Omega)}^{\|\cdot\|_{\bs{L}^2(\Omega)}}, 
    \ \omega_0 \in L^\infty(\Omega), \ k_0\in L^1(\Omega),
\\[1.5mm]
   \text{\it there exist positive }\omega_*,\ \omega^* \text{ \it such
     that }    \omega_*\le\omega_0(x)\le\omega^*\text{ \it for a.a. }
        x\in\Omega,
\\[1.5mm]
       \text{\it there exist positive } k_* \text{ \it such
     that }    k_0(x) \ge k_* \text{ \it for a.a. }
        x\in\Omega .
       \end{array}\right\}
\end{equation}
The following theorem is the main result of our paper.
\medskip

\begin{theorem}[Main existence result] \label{th:MainExist} Assume \eqref{3.1}
  and $\alpha_2=\mathrm{const}>0$ (cf.\ \eqref{1.4}).  Then there exists a
  triple of measurable functions $\{\bs{u},\omega,k\}$ in $Q$ and a
  non-negative defect measure 
  $\mu\in \mathcal{M}_\geq (\overline{Q})$ such that
\begin{equation}
 \label{3.2}
 \frac{\omega_*}{1+t\omega_*}\le\omega(x,t)\le\frac{\omega^*}{1+t\omega^*}
 \text{ and } \frac{k_*}{(1+t\omega^*)^{\alpha_2}}\le k(x,t)\;\text{ \it for a.a.\ } (x,t)\in Q;
\end{equation}
\begin{equation}
 \label{3.3}
 \left.\begin{array}{l}
        \bs{u}\in C_{\mathrm{w}} \big([\,0,T\,];\bs{L}^2(\Omega)\big)\cap
        L^2(0,T;\bs{W}_{\mathrm{per},\diver}^{1,2}(\Omega)\big),\\[2.5mm] 
        \omega\in C_{\mathrm{w}}\big([\,0,T\,];L^2(\Omega)\big)\cap
        L^2\big(0,T;W_{\mathrm{per}}^{1,2}(\Omega)\big),\\[2.5mm] 
        k\in L^\infty\big(0,T;L^1(\Omega)\big)\cap\bigcap\limits_{1\le
          p<2}L^p\big(0,T;W_{\mathrm{per}}^{1,p}(\Omega)\big); 
        \end{array}\right\}
\end{equation}
\begin{equation}
 \label{3.5}
 \int_Q k\,\big(\big|\bs{D}(\bs{u})\big|^2 {+}|\nabla\omega|^2\big) 
  \dd x\dd t<\infty, 
\end{equation}
\begin{equation}
 \label{3.6}
 \left.\begin{array}{l}
        \displaystyle \bs{u}'
  \in\bigcap\nolimits_{ \sigma > 16/5} L^{4/3}\big(0,T;\big(
  \bs{W}_{\mathrm{per},\diver}^{1,\sigma} (\Omega)\big)^*\big),
\\[0.9em]
        \displaystyle \omega' 
  \in\bigcap\nolimits_{\sigma> 16/5} L^{4/3}\big(0,T; 
       \big(W_{\mathrm{per}}^{1,\sigma}(\Omega)\big)^*\big).
       \end{array}\right\}
\end{equation}
The triple $\{\bs u, k, \omega\}$ is a weak solution of
\eqref{eq:I.1-4} and \eqref{eq:I.IBC} in the sense of Definition
\ref{def:WeakSol} with 
\begin{equation}
 \label{3.10}
 \bs{u}(0)=\bs{u}_0 \text{ in }  \bs{L}^2(\Omega) \  
 \text{ and } \
 \omega(0)=\omega_0 \text{ in } L^2(\Omega); 
\vspace{0.6em}
\end{equation}
In particular, \eqref{2.6} holds and for all $\sigma> 16/5$ we have 
\begin{equation}
 \label{3.7}
 \left.\begin{array}{l}
        {\displaystyle\int_0^T\!\!\big\langle\bs{u}'(t),\bs{v}(t) 
         \big\rangle_{W_{\mathrm{per},\diver}^{1,\sigma}} \!\!\dd t
          +\int_Q\Big( {-}(\bs{u} {\otimes} \bs{u}) {:} \nabla\bs{v}
  +\nu_0\frac k\omega\,\bs{D}(\bs{u}) 
  {:}\bs{D}(\bs{v})\Big)\dd x\dd t }
\\[0.7em]
        {\displaystyle=\int_Q\bs{f}{\cdot}\bs{v}\,\dd x\dd t}
     \qquad \ \text{ for all } \bs{v}\in
 L^\sigma \big(0,T; \bs{W}_{\mathrm{per},\diver}^{1,\sigma}(\Omega)\big) ; 
  \end{array}\right\}
\vspace{0.6em}
\end{equation}
\begin{equation}
 \label{3.8}
 \left.\begin{array}{l}
     \displaystyle
  \int_0^T\big\langle\omega'(t),\varphi(t) 
               \big\rangle_{W_{\mathrm{per}}^{1,\sigma}}\dd t
   -\int_Q\omega\bs{u}\cdot\nabla\varphi\,\dd x\dd t+\int_Q\frac
     k\omega\,\nabla \omega\cdot\nabla\varphi\,\dd x\dd t \\
 \displaystyle=-\int_Q\omega^2\varphi\,\dd x\dd t
  \quad \text{ for all }
    \varphi\in L^\sigma \big(0,T; W_{\mathrm{per}}^{1,\sigma} 
     (\Omega)\big) . 
       \end{array}\right\}
\end{equation}
\end{theorem}

Of course, in \eqref{3.7} and \eqref{3.8} it suffices to consider
$\sigma = \frac{16}5+\eta$ for an arbitrarily small $\eta>0$. 
The derivatives $\bs{u}'$ and $\omega'$ in \eqref{3.6} are understood
in the sense of distributions from $]\,0,T\,[\,$ into
$\big(\bs{W}_{\mathrm{per},\diver}^{1,\sigma}(\Omega)\big)^*$ and
$\big(W_{\mathrm{per}}^{1,\sigma}(\Omega)\big)^*$, respectively (see e.g.\
\cite[App.]{Ref5}, \cite[pp.\,54--56]{Ref9} for details). Here
we have used the continuous and dense embeddings 
\[
 W_{\mathrm{per}}^{1,2}(\Omega)\subset
 L^2(\Omega)\subset\big(W_{\mathrm{per}}^{1,\sigma}(\Omega)\big)^*\quad
 \text{for } { \sigma \geq  \frac{6}5. }
\]

To see that $\{\bs{u},\omega,k\}$ together with the measure $\mu$ in
the above theorem are a weak solution of \eqref{eq:I.1-4}  and
\eqref{eq:I.IBC}  in the sense of the Definition
\ref{def:WeakSol}, it suffices to note that \eqref{2.4}
and \eqref{2.5} follow from \eqref{3.7} and \eqref{3.8}, respectively,
by integration by parts of the first integrals on the left-hand sides.

Before starting the proof it is 
instructive to check that the above estimates \eqref{3.2} to
\eqref{3.6} are enough to show that all terms in \eqref{3.7}, \eqref{3.8}, and
\eqref{2.6} are well defined. For this, we first recall the classical
Gagliardo-Nirenberg estimate and then provide an anisotropic
  version that is adjusted to the parabolic problems on $Q=[0,T]\times
  \Omega$, we use the short-hand notations
\[
L^s(L^p):=L^s(0,T;L^p(\Omega)) \quad \text{and} \quad 
J_\theta(a,b) := a^{1-\theta}\big( a{+}b \big)^\theta.
\] 

\begin{lemma}[Gagliardo-Nirenberg estimates] \label{lem:GN}
For $N\in \mathbb N$ consider a bounded Lipschitz domain 
$\Omega\subset\mathbb{R}^N$. \\
(A) (Classical isotropic version)  Assume $1\leq p_1<
p<\infty$, 
$p_2\in {]1,N[}$ and $\theta\in {]0,1[}$ such that
\begin{equation}
  \label{eq:Interp.p}
  \frac1p = (1{-}\theta) \,\frac1{p_1}
+\theta\:\big( \frac1{p_2}-\frac1N\big). 
\end{equation} 
Then, there exists a constant $C>0$ such that for all $\psi\in W^{1,p_2}(\Omega)$ we have
\begin{equation}
 \label{3.15}
 \| \psi\|_{L^p (\Omega)} \leq C \, J_\theta\big(
 \|\psi\|_{L^{p_1}(\Omega)},  \|\nabla
\psi\|_{L^{p_2}(\Omega)} \big). 
\end{equation}
(B) (Anisotropic version) Consider $p,\;p_1,\;p_2$, and $\theta$ as in
(A) and $s,\;s_1$, and $s_2$ satisfying
\begin{equation}
  \label{eq:Interp.s}
1\leq s_2 \leq s\leq s_1 \quad \text{ and } \frac1s = (1{-}\theta) \,\frac1{s_1}
+\theta\:\frac1{s_2}.
\end{equation} 
Then, there exists $C^*>0$ such that for all $\varphi \in
L^{s_2}(0,T;W^{1,p_2}(\Omega))$ we have 
\begin{equation}
  \label{eq:GaNi.aniso}
\|\varphi\|_{L^s(L^p)} \leq C^* J_\theta \Big( 
\| \varphi\|_{L^{s_1}(L^{p_1})} , \| \nabla\varphi\|_{L^{s_2}(L^{p_2})} \Big) .
\end{equation}
\end{lemma}
\begin{proof} Part (A) is well-known, see  e.g.\
  \cite[Thm.\,1.24]{Roub13NPDE}.

To establish Part (B) we apply Part (A) for $\psi=\varphi(t)$ 
a.a.\ $t\in [0,T]$. Thus, we obtain (abbreviating
$\|\psi\|_p:=\|\psi\|_{L^p(\Omega)}$)  
\begin{align*}
\|\varphi\|^s_{L^s(L^p)}&=\int_0^T \| \varphi(t)\|_{p}^s\dd t
\overset{\text{\eqref{3.15}}}\leq C_1 \int_0^T \|\varphi(t)\|_{p_1}^{(1-\theta)s} 
\big(\|\varphi(t)\|_{p_1}{+} \|\nabla \varphi(t)\|_{p_2}
\big)^{\theta s} \dd t 
\\ 
&\hspace{-1.8em}\overset{\text{H\"older+\eqref{eq:Interp.s}}}\leq C_1
\big\|\,\|\varphi\|_{p_1}\big\|_{L^{s_1}(0,T)}^{(1-\theta)s}
\;\Big\|\, \|\varphi\|_{p_1} {+}
\|\nabla\varphi\|_{p_2} \Big\|_{L^{s_2}(0,T)}^{\theta s}\\
&\hspace{-0.6em} \overset{s_1\geq s_2}\leq  C_1
\big\|\varphi\|_{L^{s_1}(L^{p_1})}^{(1-\theta)s}\big( 
T^{1/s_2-1/s_1} \|\varphi\|_{L^{s_1}(L^{p_1})} {+}
\|\nabla\varphi\|_{L^{s_2}(L^{p_2})} \big)^{\theta s}  \\
& \leq \ C_2
\Big(  J_\theta \big( \|\varphi\|_{L^{s_1}(L^{p_1})},
\|\nabla\varphi\|_{L^{s_2}(L^{p_2})} \big) \Big)^s , 
\end{align*}
which is the desired estimate. 
\end{proof}

\begin{remark}[Well-definedness of nonlinear terms]\label{r3.1} \slshape
We first show that the second integral on the left-hand side of
the variational identity in \eqref{3.7} is  well-defined. 
For the integral of  $(\bs u {\otimes} \bs u) {:} \nabla \bs v$
we see that \eqref{3.3} allows us to use Lemma \ref{lem:GN} with $N=3$,
$(s_1,p_1)=(\infty,2)$ and $(s_2,p_2)=(2,2)$. With $\theta=3/4$ part
(A) gives 
\begin{equation}
 \label{3.13}
\|\bs u\|_{\bs L^4(\Omega)} \leq C\Big( \|\bs u\|_{\bs L^2(\Omega)} + 
 \|\bs u\|_{\bs L^2(\Omega)}^{1/4}  \|\nabla\bs u\|_{\bs L^2(\Omega)}^{3/4}\Big),
\end{equation}
whereas part (B) leads to $\bs u\in L^{8/3}(0,T;\bs L^4(\Omega))$, which implies 
\begin{equation}
 \label{3.14}
\bs u {\otimes} \bs u \in L^{4/3}(0,T;L^2(\Omega)). 
\end{equation}
With
$\sigma>16/5>2$ we have $\nabla \bs v\in L^2(0,T; L^2(\Omega))$ and
$\int_Q (\bs u {\otimes} \bs u) : \nabla \bs v\dd x \dd t$ 
is well defined.  Using $\theta=3/5$ in Lemma \ref{lem:GN}(B) we
obtain $s=p=10/3$ and hence conclude 
\begin{equation}
\label{eq:bs.u.10/3}
\| \bs u\|_{L^{10/3}(Q)} \leq C_2J_{3/5}\Big(\| \bs u\|_{L^\infty(\bs
  L^2)}, \|\nabla \bs u\|_{L^2(L^2)}\Big).  
\end{equation}

For the integral of $\frac k\omega \bs D(\bs u){:} \bs D(\bs v)$ we use
$\omega\geq 
\omega_*/(1{+}T\omega_*)>0$ from \eqref{3.2}, $k^{1/2}\bs D(\bs u)\in
L^2(Q)$  from \eqref{3.5}. Using \eqref{3.3} we can apply
Lemma \ref{lem:GN}(B) to $k$ with $N=3$, $(s_1,p_1)=(\infty,1)$, and
$s_2=p_2\in {[1,2[}$. Choosing $\theta=3/4$ we obtain $s=p=4p_2/3$,
such that $k$ lies in
$L^{4p_2/3}(0,T;L^{4p_2/3}(\Omega))=L^{4p_2/3}(Q)$. As $p_2\in {[1,2[}$ is
arbitrary, we have $k^{1/2} \in L^q(Q)$ for all $q\in {[1,16/3[}$. 
By H\"older's inequality we arrive at 
\begin{equation}
  \label{3.16}
  k\bs D(\bs u) = k^{1/2} \,k^{1/2}\bs D(\bs u) \in L^{\overline p}(Q)
  \ \text{ for all }\overline p \in {[1,16/11[}.
\end{equation}
Using $\bs D(\bs v)\in L^\sigma(0,T;L^\sigma(\Omega)) = L^{\sigma}(Q)$ 
with $ \sigma >16/5 $ we see that there is always a
$\overline p\in {[1,16/11[}$ such that
$\frac1{\sigma}+ \frac1{\overline p}\leq 1$. Hence we conclude
\[
\int_Q \big| \frac k\omega \bs D(\bs u){:} \bs D(\bs v) \big| \dd x \dd t 
 \leq C 
\| k\bs D(\bs u)\|_{L^{\overline p}(Q)} \|\bs D(\bs v) \|_{L^{\sigma}(Q)} <\infty. 
\]
Thus, by a routine argument, \eqref{3.14} and \eqref{3.16} lead to the
existence of the distributional derivative $\bs{u}'$ as in
\eqref{3.6}, see also Sections~\ref{s5.4}--\ref{s5.6}. 

An analogous reasoning applies to the second and the third integral on
the left-hand side of the variational identity in \eqref{3.8}.

Finally, combining $\bs u \in \bs L^2(Q)$ and $\nabla k\in L^p(Q)$ for
all $p\in {[1,2[}$ (see \eqref{3.3}) and $k \in
L^{4p/3}(Q)$ from above,  H\"older's inequality gives 
\[
 k \bs u \in \bs L^q(Q) \text{ and } k \nabla k\in \bs
 L^q(Q)\quad\text{for all } q\in {[1,8/7[}, 
\]
i.e., the second and third integral on the left-hand side in
\eqref{2.6} are well defined.
\end{remark}

The estimates \eqref{3.2}, which will be derived by using
suitable comparison arguments, allow us to deduce the following
result (based on the choice $\alpha_1=1$). 

\begin{corollary}\label{c3.1}
 For a.a.\ $(x,t)\in Q$, we have the following estimates:
 \begin{align}
  \label{3.17}
  &L(x,t):=\frac{k(x,t)^{1/2}}{\omega(x,t)}\ge
  \frac{k_*^{1/2}}{\omega^*}\:(1+t\omega^*)^{1-\alpha_2/2},\\ 
  \label{3.18}
  &\frac 1{\omega^*}+t\le\frac1{\omega(x,t)}\le\frac1{\omega_*}+t.
 \end{align}
\end{corollary}

\noindent
Kolmogorov claimed in \cite{Ref14} that $L=L(x,t)$ ``... grows in proportion of
$t^{2/7}$ ...'' (see also \cite[p.\,215]{Ref25},
\cite[p.\,329]{Ref27}). Clearly, from \eqref{3.17} with
$\alpha_2=10/7$ it follows
\[
L(x,t)\ge\frac{k_*^{1/2}}{\omega^*}(1+t\omega^*)^{2/7}\quad\text{for a.a. }\; (x,t)\in\Omega\times\,]\,t_0,T\,[\,.
\]
Of course, Kolmogorov's claim is compatible with our lower estimate for any
choice $\alpha_2\geq 10/7$ (and in \cite{Ref14} $\alpha_2=11/7$ was
chosen). However, it cannot be true for $\alpha_2 \in {]0,10/7[}$.


\section{Proof of the existence theorem}\label{s5}

The proof of the main Theorem \ref{th:MainExist} proceeds in several
steps. First we regularize the problem by adding small higher-order
dissipation terms of $r$-Laplacian type and small
coercivity-generating lower order terms. A general result for
pseudo-monotone operators, which is detailed in Appendix \ref{s.A},
then provides approximate solutions $\{\bs u_\eps,
\omega_\eps,k_\eps\}$. In Section \ref{s5.2} we provide
$\eps$-independent upper and lower bounds for $\omega_\eps$ and
$k_\eps$ by comparison arguments. In Section \ref{s5.3} we
complement the standard energy estimates by improved integral
estimates for $k_\eps$ that allow us to pass to the limit $\eps
\searrow 0$ in Section \ref{s5.5}.

\subsection{Defining suitable approximate solutions 
$\{ \bs{u}_\eps,  \omega_\eps, k_\eps\}$ }
\label{s5.1}

Let be $\omega_*$, $\omega^*$ and $k_*$ as in \eqref{3.1}. We
introduce the comparison functions 
\begin{equation}
 \label{4.1}
 \underline{\omega}(t)=\frac{\omega_*}{1+t\omega_*},\quad\overline{\omega}(t)=\frac{\omega^*}{1+t\omega^*},\quad\kappa(t)=\frac{k_*}{(1+t\omega^*)^{\alpha_2}}\quad
 \text{for } t\in[0,T],
\end{equation}
which will be the desired bounds for $\omega_\eps$ and
$k_\eps$ in $Q$. Subsequently we will use the notion 
\[
\xi^+ := \max\{\xi,0\} \geq 0 \quad \text{and} \quad
\xi^-=\min\{\xi,0\} \leq 0
\]
for the positive and negative parts of real numbers or real-valued
functions.

We choose a fixed number $r\in {]3,\infty[}$ and consider for all small
$\eps >0$ the following {\it $r$-Laplacian approximation\/} of
\eqref{eq:I.1-4}, where we add the coercivity-generating terms 
$\eps\big|\bs u\big|^{r-1} \bs u$,  $\eps |\omega|^{r-2} \omega $
and $\eps |k|^{r-2} k$ to the right-hand sides of \eqref{1.2} to 
\eqref{1.4}, respectively:
\begin{subequations}
\label{eq:4Approx}
\begin{align}
 \label{4.2}
&\hspace*{5em} \diver\bs{u}=0,
\\[0.7em]
\label{4.3}
&\begin{aligned}
      \frac{\partial\bs{u}}{\partial t}+(\bs{u}\cdot\nabla)\bs{u} 
& = \nu_0\diver\Big(\frac{k^+}{\eps+\omega^+}\,\bs{D}(\bs{u})\Big) 
     -\nabla p+\bs{f}\\[0.3em] 
     &\quad
     +\eps\Big(\diver\big(\big|\bs{D}(\bs{u})\big|^{r-2} 
               \bs{D}(\bs{u})\big)-|\bs{u}|^{r-2}\bs{u}\Big), 
      \end{aligned}
\\[0.7em]
 \label{4.4}
 &\begin{aligned}
       \frac{\partial\omega}{\partial t}+\bs{u}\cdot\nabla\omega
&= \diver\Big(\frac{k^+}{\eps+\omega^+}\,\nabla\omega\Big)-\omega^+\omega
  \\[0.3em] 
        &\quad+\eps\Big(\diver\big(|\nabla\omega|^{r-2}\nabla\omega\big) 
       -|\omega|^{r-2}\omega\Big)+\eps\big(\underline{\omega}(t)\big)^{r-1},
       \end{aligned}
 \\[0.7em]
 \label{4.5}
 &\begin{aligned}
     \frac{\partial k}{\partial t}+\bs{u}\cdot\nabla k
    &=\diver\Big(\frac{k^+}{\eps\!+\!\omega^+}\,\nabla
    k\Big)\!+\nu_0\frac{k^+}{\eps+\omega^+\!+\eps
      k^+}\,\big|\bs{D}(\bs{u})\big|^2\!\!-\!\alpha_2 k\omega^+
     \\[0.3em] 
        &\quad+\eps\Big(\diver\big(|\nabla k|^{r-2}\nabla k\big)- 
       |k|^{r-2}k\Big)+\eps\big(\kappa(t)\big)^{r-1}.        
       \end{aligned}
\end{align}
\end{subequations}
The additional terms $\eps\big(\underline{\omega}(t)\big)^{r-1}$ and
$\eps\big(\kappa(t)\big)^{r-1} $ are added in \eqref{4.4} and \eqref{4.5},
respectively, to make the comparison principle work again.   
In principle, it would be possible to use different exponents $r_{\bs u}$,
$r_\omega$, and $r_k$ in the equations \eqref{4.3} to \eqref{4.5}, because they need
to satisfy different restrictions. In our case $r=r_{\bs u}=r_\omega=r_k$ is sufficient
and fits exactly with the assumptions in \eqref{eq:A.Assum} with $p=r$ for the
abstract existence Theorem \ref{thm:A}. 

We consider system \eqref{eq:4Approx} with initial data
$\{\bs{u}_{0,\eps},\omega_{0,\eps},k_{0,\eps}\}$ satisfying 
\begin{subequations}
\label{eq:4.IC}
\begin{align}
 \label{4.6}
 &\{\bs{u}_{0,\eps},\omega_{0,\eps},k_{0,\eps}\}\in\bs{W}_{\mathrm{per},\diver}^{1,r}
  (\Omega)\times  W_{\mathrm{per}}^{1,r}(\Omega)\times W_{\mathrm{per}}^{1,r}(\Omega) ,
\\[0.5em]
\label{4.7}
 &\omega_*\le\omega_{0,\eps}(x)\le\omega^* \ \text{ and } \  k_{0,\eps}(x)\ge
 k_*\quad\text{ a.e.\ in } \Omega,
\\[0.5em]
\label{4.8}
 &\left.\begin{array}{@{}l}
          \bs{u}_{0,\eps}\longrightarrow \bs{u}_0\text{ in }\bs{L}^2(\Omega),\quad \omega_{0,\eps}\longrightarrow \omega_0\text{ a.e.\ in } \Omega,\\[2mm]
          k_{0,\eps}\longrightarrow k_0\;\text{ in }\; L^1(\Omega)\ 
     \text{ for }\eps\rightarrow 0.
         \end{array}\right\}
\end{align}
\end{subequations}
The existence of a sequence $\{\bs{u}_{0,\eps}\}_{\eps>0}$ which
satisfies \eqref{4.6} follows immediately from the condition on
$\bs{u}_0$ in \eqref{3.1}, whereas the existence of sequences
$\{\omega_{0,\eps}\}_{\eps>0}$ and $\{k_{0,\eps}\}_{\eps>0}$
satisfying \eqref{eq:4.IC} can be derived by routine argument
from the conditions on $\omega_0$ and $k_0$ in \eqref{3.1}.

The following lemma states the existence of weak solutions of 
\eqref{eq:4Approx} under the periodic boundary conditions
\eqref{1.5} and initial data \eqref{eq:4.IC}. This result, which
we derive in Appendix \ref{s.A} by a direct application of existence
results for pseudo-monotone evolutionary problems (see Theorem
\ref{thm:A}), forms the starting point for our discussion in
Subsections~\ref{s5.2}--\ref{s5.6}.


\begin{proposition}[Existence of approximate solutions] \label{p4.1}
  Let $\{\bs{u}_{0,\eps},\omega_{0,\eps},k_{0,\eps}\}_{\eps>0}$ be as
  in \eqref{eq:4.IC}, $r>3$, and $\bs{f}\in\bs{L}^2(Q)$.
  Then, for every $\eps>0$ there exists a triple
  $\{\bs{u}_\eps,\omega_\eps,k_\eps\}$ such that
\begin{subequations}
 \label{4.9}
\begin{align}
 \label{4.9a}
        &\bs{u}_\eps\in C\big([\,0,T\,];\bs{L}^2(\Omega)\big)\cap
        L^r\big(0,T;\bs{W}_{\mathrm{per},\diver}^{1,r}(\Omega)\big),
 \\[1.5mm]
 \label{4.9b}     
   &\omega_\eps,k_\eps\in C\big([\,0,T\,];L^2(\Omega)\big)\cap
   L^r\big(0,T;W_{\mathrm{per}}^{1,r}(\Omega)\big), 
\\[1.5mm]
 \label{4.9c}
 &\bs{u}'_\eps \in
 L^{r'}\big(0,T;\big(\bs{W}_{\mathrm{per},\diver}^{1,r}(\Omega)\big)^*\big),
 \quad \omega'_\eps,\, k'_\eps\in
 L^{r'}\big(0,T;\big(W_{\mathrm{per}}^{1,r}(\Omega)\big)^*\big),
\end{align}
\end{subequations}
and 
\begin{subequations}
\label{eq:ApproxSys}
\begin{equation}
 \label{4.11}
 \left.\begin{array}{l}
        \displaystyle \int_0^T\big\langle\bs{u}'_\eps(t),\bs{v}(t)\big\rangle_{W_{\mathrm{per},\diver}^{1,r}}\dd
          t+\int_Q \sum_{i=1}^3 u_{\eps,i}(\partial_i\bs{u}_\eps)
          \cdot\bs{v}\dd x\dd t \\ 
       \quad +{\displaystyle\,\nu_0\int_Q\frac{k_\eps^+}{\eps+\omega_\eps^+}\,\bs{D}(\bs{u}_\eps):\bs{D}(\bs{v})\dd x \dd t}\\
      \quad  +{\displaystyle\,\eps\int_Q\big(\big|\bs{D}(\bs{u}_\eps)\big|^{r-2}\bs{D}(\bs{u}_\eps):\bs{D}(\bs{v})+|\bs{u}_\eps|^{r-2}\bs{u}_\eps\cdot\bs{v}\big)\dd x\dd t}\\
        ={\displaystyle\int_Q\bs{f}\cdot\bs{v}\dd x \dd t\qquad\quad 
 \text{ for all  } \bs{v}\in L^r\big(0,T; \bs{W}_{\mathrm{per},\diver}^{1,r}(\Omega)\big),}        
       \end{array}\right\}
\vspace{0.6em}
\end{equation}
\begin{equation}
 \label{4.12}
 \left.\begin{array}{l}
        {\displaystyle\int_0^T\big\langle\omega'_\eps(t),\varphi(t)\big\rangle_{W_{\mathrm{per}}^{1,r}}\dd t+\int_Q\varphi\bs{u}_\eps\cdot\nabla\omega_\eps\dd x\dd t}\\
    \quad    {\displaystyle+\int_Q\frac{k_\eps^+}{\eps+\omega_\eps^+}\,\nabla\omega_\eps\cdot\nabla\varphi\dd x\dd t+\int_Q \omega_\eps^+\omega_\eps\varphi\dd x\dd t}\\
     \quad   {\displaystyle+\,\eps\int_Q\big(|\nabla\omega_\eps|^{r-2}\nabla\omega_\eps\cdot\nabla\varphi+|\omega_\eps|^{r-2}\omega_\eps\varphi\big)\dd x\dd t}\\
        ={\displaystyle\eps\int_Q\big(\underline{\omega}(t)\big)^{r-1}\varphi\dd
          x\dd t \qquad\  \text{for all }\varphi\in L^r\big(0,T;W_{\mathrm{per}}^{1,r}(\Omega)\big),}
       \end{array}\right\}
\vspace{0.6em}
\end{equation}
\begin{equation}
 \label{4.13}
 \left.\begin{array}{l}
        {\displaystyle\int_0^T\big\langle k'_\eps(t),z(t)\big\rangle_{W_{\mathrm{per}}^{1,r}}\dd t+\int_Q z\bs{u}_\eps\cdot\nabla k_\eps\dd x\dd t}\\
       \quad
       {\displaystyle+\int_Q\frac{k_\eps^+}{\eps{+}\omega_\eps^+}\,\nabla
         k_\eps {\cdot}\nabla z\dd x\dd t
  -\nu_0\int_Q \frac{k_\eps^+}{\eps{+}\omega_\eps^+{+}\eps k_\eps^+}
  \, \big|\bs{D}(\bs{u}_\eps)\big|^2z\dd x\dd t}\\
     \quad    {\displaystyle+\,\alpha_2\int_Q k_\eps\omega_\eps^+ z \dd x\dd t
        +\eps\int_Q\!\! \big(|\nabla k_\eps|^{r-2}\nabla
        k_\eps{\cdot}\nabla z {+}|k_\eps|^{r-2}k_\eps z\big)\dd x\dd t}\\
        {\displaystyle=\eps\int_Q \big(\kappa(t)\big)^{r-1}z\dd x\dd t
    \qquad \qquad \text{for all } z\in
    L^r\big(0,T;W_{\mathrm{per}}^{1,r}(\Omega)\big),} 
       \end{array}\right\}
\vspace{0.8em}
\end{equation}
\end{subequations}
\begin{equation}
 \label{4.14}
 \bs{u}_\eps(0)=\bs{u}_{0,\eps},\quad\omega_\eps(0)=\omega_{0,\eps},\quad k_\eps(0)=k_{0,\eps}.
\end{equation}
\end{proposition}

The proof of Proposition~\ref{p4.1} is the content of Appendix \ref{s.A}.
Observing the separability of $\bs{W}_{\mathrm{per},\diver}^{1,r}(\Omega)$ and
$W_{\mathrm{per}}^{1,r}(\Omega)$ and using \eqref{4.9}, a routine argument
yields that the system \eqref{eq:ApproxSys} is equivalent to the following
conditions for a.a.\ $t\in [0,T]$:
\begin{subequations}
\label{eq:AppSyst.a.a.}
\begin{equation}
  \label{4.15}
  \left.\begin{array}{l}
    {\displaystyle\big\langle\bs{u}'_\eps(t),\bs{w}\big\rangle_{W_{\mathrm{per},\diver}^{1,r}}
    +\int_\Omega \!\! \Big(\big(\bs{u}_\eps(t){\cdot}\nabla
         \bs{u}_\eps (t) \big) 
   {\cdot }\bs{w} +\nu_0 \frac{k_\eps^+(t)}{\eps{+}\omega_\eps^+(t)}\,\bs{D}\big(\bs{u}_\eps(t)\big){:}\bs{D}(\bs{w})\Big)\dd x}\\[0.5em]
      \quad    {\displaystyle+\,\eps\int_\Omega\big(\big|\bs{D}\big(\bs{u}_\eps(t)\big)\big|^{r-2}\bs{D}\big(\bs{u}_\eps(t)\big):\bs{D}(\bs{w})+\big|\bs{u}_\eps(t)\big|^{r-2}\bs{u}_\eps(t)\cdot\bs{w}\big)\dd x}\\[0.5em]
        {\displaystyle=\int_\Omega\bs{f}(t)\cdot\bs{w}\dd
          x\qquad\qquad \text{for all }\bs{w}\in\bs{W}_{\mathrm{per},\diver}^{1,r}(\Omega),}
 \end{array}\right\}
 \end{equation}
\begin{equation}
 \label{4.16}
 \left.\begin{array}{l}
 \displaystyle\big\langle\omega'_\eps(t),\psi\big\rangle_{W_{\mathrm{per}}^{1,r}}
          + \int_\Omega 
          \!\!\Big( \psi\bs{u}_\eps(t)\cdot\nabla\omega_\eps(t)+
          \frac{k_\eps^+(t)}{\eps{+} 
            \omega_\eps^+(t)}\,\nabla\omega_\eps(t){\cdot}\nabla\psi
          \Big) \dd x \\[0.5em]
 \quad \displaystyle +\int_\Omega\!\!\Big( \omega_\eps^+(t)\omega_\eps(t)\psi
        +\eps \big(\big|\nabla\omega_\eps(t)\big|^{r-2}
          \nabla\omega_\eps(t)\cdot\nabla\psi
          +\big|\omega_\eps(t)\big|^{r-2} \omega_\eps(t)\psi\big) \Big)
          \dd x \\[0.5em]
        {\displaystyle=\eps\big(\underline{\omega}(t)\big)^{r-1}\int_\Omega\psi\dd
          x \qquad\qquad \text{for all } \psi\in W_{\mathrm{per}}^{1,r}(\Omega),}
       \end{array}\right\}
\end{equation}
\begin{equation}
 \label{4.17}
 \left.\begin{array}{l}
        {\displaystyle\big\langle
          k'_\eps(t),z\big\rangle_{W_{\mathrm{per}}^{1,r}}+\int_\Omega \!\Big(
          z\bs{u}_\eps(t)\cdot\nabla k_\eps(t) +
          \frac{k_\eps^+(t)}{\eps+\omega_\eps^+(t)}\,\nabla
          k_\eps(t)\cdot\nabla z \Big) \dd x}
  \\[0.8em]
        \;\;{\displaystyle-\nu_0\int_\Omega
          \frac{k_\eps^+(t)}{\eps{+}\omega_\eps^+(t) {+}\eps k_\eps^+(t)} 
         \,\big|\bs{D}\big(\bs{u}_\eps(t)\big)\big|^2z\dd x  +
        \alpha_2\int_\Omega k_\eps(t)\omega_\eps^+(t) z\dd x}
 \\[0.8em]
        \;\;{\displaystyle+\,\eps\int_\Omega\big(\big|\nabla k_\eps(t) 
          \big|^{r-2}\nabla k_\eps(t)\cdot\nabla z+\big|k_\eps(t) 
                 \big|^{r-2}k_\eps(t)z\big)\dd x}
  \\[0.8em]
        {\displaystyle=\eps\big(\kappa(t)\big)^{r-1}\int_\Omega z\dd x
         \qquad\qquad \text{for all }  z\in W_{\mathrm{per}}^{1,r}(\Omega)}
       \end{array}\right\}
\end{equation}
\end{subequations}
We notice that the set
$\mathcal{N}\subset[\,0,T\,]$ of measure zero of those $t$ where
\eqref{eq:AppSyst.a.a.} fails, does not depend on
$(\bs{w},\psi,z)$. More specifically, if $\eps=\eps_m>0$ with
$\lim\limits_{m\to\infty}\eps_m=0$, then 
$\mathcal{N}$ can be chosen independently of $m$.

The variational identities in \eqref{eq:AppSyst.a.a.} are the
point of departure for the proof of a series of the a priori
estimates for 
$\{\bs{u}_\eps,\omega_\eps,k_\eps\}$ we are going to derive in
Subsections~\ref{s5.2}--\ref{s5.4}.
%

\subsection{Upper and lower bounds for
  $\{\omega_\eps,k_\eps\}$}\label{s5.2} 

Let $\underline{\omega}$, $\overline{\omega}$ and $\kappa$ be as in
\eqref{4.1} and $r>3$ as chosen in Section \ref{s5.1}. The
following result provides pointwise upper and lower bounds that are
obtained via classical comparison arguments for weak solutions of
the scalar parabolic equations for $\omega$ and $k$, cf.\ \eqref{1.3} and
\eqref{1.4}, respectively.  


\begin{lemma}\label{l4.2}
 Let be $\{\bs{u}_\eps,\omega_\eps,k_\eps\}$ a
 triple according to Proposition~\ref{p4.1} with $r>3$. Then, 
 \begin{equation} 
  \label{4.18}
  \underline{\omega}(t)\le\omega_\eps(x,t)\le\overline{\omega}(t) 
\quad \text{ and } \quad \kappa(t)\le k_\eps(x,t)
 \end{equation}
for a.a.\ $(x,t)\in Q$ and for all $\eps>0$.
\end{lemma}
\begin{proof}
 For notational simplicity, we set
 $\bs{u}\equiv\bs{u}_\eps$, $\omega\equiv\omega_\eps$ and $k\equiv
 k_\eps$ within this proof.

\underline{Step 1: $\omega \geq \underline\omega$.}  The function
$\psi=\big(\omega(\cdot,t)-\underline{\omega}(t)\big)^-$ is an admissible
test function for \eqref{4.16}. Since $\underline{\omega}(t)$
does not depend on $x$ 
we have $\frac12\nabla (\psi^2)=\psi\nabla \omega$ and $\nabla
\omega \cdot \nabla \psi=|\nabla \psi|^2\geq 0$. Using $
\underline{\omega} >0$ and the
monotonicity of $\omega \mapsto |\omega|^{r-2}\omega$ we arrive at 
\begin{align}
\nonumber
&  \big\langle\omega'(t),\big(\omega(t)-\underline{\omega}(t)\big)^- 
 \big\rangle_{W_{\mathrm{per}}^{1,r}}+
\int_\Omega\omega^2\big(\omega-\underline{\omega}(t)\big)^-\dd x\\ 
\label{4.19}
& \le
\eps\int_\Omega\big(\big(\underline{\omega}(t)\big)^{r-1}-
  |\omega|^{r-2} \omega \big)
\big(\omega-\underline{\omega}(t)\big)^-\dd
x \qquad 
\le 0
\end{align}
for a.a.\ $t\in[\,0,T\,]$. By construction we have
$\underline{\omega}'(t)=\frac{\mathrm d}{\mathrm d
  t}\underline\omega(t) =-\big(\underline{\omega}(t)\big)^2$. 
Identifying $\underline{\omega}$ with a function in
$C^1\big([0,T\,];W_{\mathrm{per}}^{1,r}(\Omega)\big)$ the estimate
\eqref{4.19} leads to 
\[
\big\langle\omega'(t)-\underline{\omega}'(t),
\big(\omega(t){-}\underline{\omega}(t)\big)^-\big\rangle_{W_{\mathrm{per}}^{1,r}}\le 
-\int_\Omega\big(\omega^2-\big(\underline{\omega}(t)\big)^2\big)\big(\omega-\underline{\omega}(t)\big)^-\dd x
\le 0.
\]
By \eqref{4.1} and \eqref{4.7}, we have
$\omega(x,0)-\underline{\omega}(0)\ge 0$, which means $\psi(x,0)=0$
for a.a. $x\in\Omega$.  Using a slight modification of
\cite[pp.\,290--291]{Ref18} we find 
\begin{align*}
\int_\Omega \frac12\psi(t)^2\dd x = \int_\Omega \frac12\psi(0)^2 \dd x
+ \int_0^t \langle \psi',\psi\rangle_{W_{\mathrm{per}}^{1,r}} \dd t =
0 + \int_0^t  \langle \omega'{-}\underline{\omega}',
(\omega{-}\underline{\omega})^- \rangle_{W_{\mathrm{per}}^{1,r}} \dd t
\leq 0.   
\end{align*}
Hence, we conclude $\psi(t)=0$ for all $t$, which means that 
\begin{equation}
 \label{4.20}
 \omega(x,t)\ge\underline{\omega}(t)\quad\text{ for a.a. }\; (x,t)\in Q.
\end{equation}

\underline{Step 2: $\omega \leq \overline\omega$.}  Next, we insert
$\psi=\big(\omega(\cdot,t){-}\overline{\omega}(t)\big)^+$ in \eqref{4.16} and
argue as in Step 1:
\[
\big\langle\omega',(\omega{-}\overline{\omega})^+
 \big\rangle_{W_{\mathrm{per}}^{1,r}}+ \int_\Omega\omega^2 
 \big(\omega-\overline{\omega}\big)^+\dd x
\le\eps \! \int_\Omega \!\big((\underline{\omega})^{r-1} 
  {-}\omega^{r-1}\big)\big(\omega{-}\overline{\omega}\big)^+\dd x 
\le 0.
\]
For the last estimate we used $\omega\ge\underline{\omega}$,
which was obtained in Step 1. Hence, as above, 
\[
\frac{\mathrm d}{\mathrm d t} \int_\Omega \frac12\psi(t)^2\dd x = 
\big\langle\omega'(t)-\dot{\overline{\omega}}(t),\big(\omega(t)-\overline{\omega}(t)\big)^+\big\rangle_{W_{\mathrm{per}}^{1,r}}
\leq -\int_\Omega\big(\omega^2 {-} \overline{\omega}^2\big)
\big(\omega {-} \overline{\omega}\big)^+\dd x \le 0
\]
for a.a. $t\in[\,0,T\,]$. Again by \eqref{4.1} and \eqref{4.7},
we have $\psi(0)=0$ a.e.\ in $\Omega$ and conclude 
\begin{equation}
 \label{4.21}
 \omega(x,t)\le\overline{\omega}(t)\quad\text{ for a.a.\ } (x,t)\in Q.
\end{equation}

\underline{Step 3: $k \geq \kappa$.}  We first insert $z=k^-(\cdot,t)$
into \eqref{4.17} and find $k\ge 0$ a.e. in $Q$. Next, we insert the
test function $z(x,t)=\big(k(x,t)-\kappa(t)\big)^-$ and obtain as
above
\[
 \big\langle k'(t),\big(k(t)-\kappa(t)\big)^-\big\rangle_{W_{\mathrm{per}}^{1,r}}+\alpha_2\int_\Omega k(t)\omega(t)\big(k{-}\kappa(t)\big)^-\dd x\le 0
\]
for a.a. $t\in[\,0,T\,]$. By construction $\kappa$ satisfies 
$\kappa'(t)=-\alpha_2\kappa(t)\overline{\omega}(t)$ for all
$t\in[\,0,T\,]$. It follows
\begin{align*}
&\frac{\mathrm d}{\mathrm d t} \int_\Omega
\frac12\Big(\big(k(t){-}\kappa(t)\big)^-\Big)^2 \dd x = 
 \big\langle k'(t)-\dot{\kappa}(t),\big(
 k(t)-\kappa(t)\big)^-\big\rangle_{W_{\mathrm{per}}^{1,r}} \\
&\le 
-\alpha_2\int_\Omega
\big(k(t)\omega(t)-\kappa(t)\overline{\omega}(t)\big)
\big(k(t)-\kappa(t)\big)^-\dd  x \ \le \ 0.
\end{align*}
To see the last inequality, we use $\omega\le\overline{\omega}$
a.e. in $Q$ from Step 2, which gives
$k(x,t)\omega(x,t)\le\kappa(t)\overline{\omega}(t)$ for a.a. $x$ of
the set $\big\{x\in\Omega\: ; \: k(x,t) \leq \kappa(t)\big\}$. Since
$k(x,0)\ge \kappa(0)$ for  a.a. $x\in \Omega$ by \eqref{4.1} and
\eqref{4.7}  we obtain, as above, $ k(x,t) \geq \kappa(t)$ for a.a.\ $(x,t)\in Q$. 
Altogether the upper and lower bounds in \eqref{4.18} are established. 
\end{proof}

\subsection{Energy estimates for $(\bs u_\eps,\omega_\eps)$ and
  improved  estimates for $k_\eps$}
\label{s5.3}

For the subsequent estimates we fix the data 
\[
\mathfrak D= \{T, \; \bs{f}, \; \omega_*, \; \omega^*, \; k_*, \; r\}
\]
and will indicate constants that only depend on $\mathfrak D$ by
$\CD$. However, depending on the context the constants $\CD$ may be
different. We also define
the constant 
\[
 \beta_*=\frac{k_*}{(1+\omega^*)(1+T\omega^*)^{\alpha_2}},
\]
which according to Lemma \ref{l4.2} is a lower bound for
$k_\eps/(\eps{+}\omega_\eps)$.  This will allows us to derive the
standard estimates for $\bs u_\eps$ and $\omega_\eps$.

\begin{lemma}\label{l4.3}
There exists a constant $\CD>0$ such for all $\eps\in
{]0,1]}$ and all solutions $\{\bs{u}_\eps,\omega_\eps,k_\eps\}$ as in
Proposition~\ref{p4.1} we have the estimates 
\begin{subequations}
\label{eq:AppSolEE}
\begin{equation}
 \label{4.22}
 \left.\begin{array}{r}
   \displaystyle\|\bs{u}_\eps\|_{L^\infty(\bs{L}^2)}^2+ \int_Q
          \!
          \big(\beta_*+\frac{k_\eps}{\eps{+}\omega_\eps}\big) 
   \big|\bs{D}(\bs{u}_\eps)\big|^2\dd x\dd t  + \eps \int_Q\!\! 
   \big(\big|\bs{D}(\bs{u}_\eps) \big|^r {+} 
|\bs{u}_\eps|^r\big)\dd x\dd t \  \\[0.7em]
       \le   \CD \big(\|\bs{u}_{0,\eps}\|_{\bs{L}^2}^2+\|\bs{f}\|_{\bs{L}^2}^2\big),
       \end{array}\right\}
\end{equation}
\begin{equation}
 \label{4.23}
 \left.\begin{array}{r}
   \displaystyle\|\omega_\eps\|_{L^\infty(L^2)}^2+\int_Q\Big(\beta_* 
   +\frac{k_\eps}{\eps{+}\omega_\eps}\Big)|\nabla\omega_\eps|^2\dd x\dd t 
        + \eps\int_Q\big(|\nabla\omega_\eps|^r 
     +\omega_\eps^r\big)\dd x\dd t \quad \\[0.8em]
        \leq \CD \big(1+\|\omega_{0,\eps}\|_{L^2}^2 \big).
       \end{array}\right\}
\end{equation}
\end{subequations}
\end{lemma}
\begin{proof}
  We insert the test functions $\bs{w}=\bs{u}_\eps$ and
  $\psi=\omega_\eps$ in \eqref{4.15} and \eqref{4.16}, respectively.
  Integrating over $[0,t]$ and using $ \frac{k_\eps}{\eps+\omega_\eps}
  \ge \beta_*$ a.e.\ in $Q$ (cf.\ \eqref{4.18}), the desired estimates 
 \eqref{eq:AppSolEE} are readily obtained by the
 aid of Gronwall's lemma.
\end{proof}

By \eqref{eq:4.IC} the approximative initial conditions satisfy
$ \sup_{0<\eps\leq 1} \big(\|\bs{u}_{0,\eps}\|_{\bs{L}^2} {+}
\|\omega_{0,\eps}\|_{L^2}\big)<\infty$.  Therefore all terms on the
left hand sides of \eqref{eq:AppSolEE} are bounded independently of
$\eps \in {]0,1]}$.

Of course, one obtains a trivial bound for $k_\eps$ in
$L^\infty(0,T;L^1(\Omega))$ by testing \eqref{4.17} with $z\equiv
1$. We include this result in the following non-trivial estimate
that implies uniform higher 
integrability of $k_\eps$ as well as suitable bounds for $\nabla
k_\eps$. For this we test \eqref{4.17} by $z=1-(1{+}k_\eps)^{-\delta}$
for $\delta \in {]0,1[}$, which is a well-known technique 
for treating diffusion equations with an $L^1$ right-hand side, see
e.g.\ \cite{BocGal89NEPE,Rako91,BoDaOr97}.

\begin{proposition}\label{pr:4.4} For given data $\mathfrak D$, $p\in
  {[1,2[}$, and $\delta \in {]0,1[}$, there exists $\CD^{p,\delta}>0$ such that for
  all $\eps \in {]0,1]}$ and all $\{\bs{u}_\eps,\omega_\eps,k_\eps\}$
  as in Proposition~\ref{p4.1}, we have the estimate
 \begin{equation}
  \label{4.24}
  \left.\begin{array}{l}
         {\displaystyle\|k_\eps\|_{L^\infty(0,T;L^1(\Omega))}+ 
 \int_Q\Big(k_\eps^{4p/3}+|\nabla k_\eps|^p+\frac{|\nabla k_\eps|^2} 
   {(1{+}k_\eps)^\delta}\Big)\dd x\dd t}
\\[0.8em]
 \hspace{7em} \displaystyle 
  + \,\eps \int_Q \Big(\frac{|\nabla k_\eps|^r}{(1{+}k_\eps)^{1+\delta}} +
 k_\eps^{r-1} \Big) \dd x \dd t 
\\[0.98em]
      \qquad   \le \CD^{p,\delta} 
  \big(1+\|\bs{u}_{0,\eps}\|_{\bs{L}^2(\Omega)}^2
        +\|k_{0,\eps}\|_{L^1(\Omega)} \big). 
          \end{array}\right\}
 \end{equation}
\end{proposition}
%
%
\begin{proof} \underline{Step 1:} 
 For $0<\delta<1$ we define $\Phi:{[0,\infty[}\to {[0,\infty[}$
 via 
 \[
  \Phi(\tau)=\tau+\frac1{1-\delta}\big(1-(1{+}\tau)^{1-\delta}\big),\quad
  0\le\tau<\infty. 
 \]
Hence, $\Phi$ is convex and satisfies, for all $\tau \geq 0$, the estimates 
\begin{equation}
 \label{4.26}
\frac\tau 2-\frac{2}{1{-}\delta}\le\Phi(\tau)\le\tau,\quad 
 \Phi'(\tau)=1-\frac1{(1+\tau)^\delta} \in [0,1],\quad
 \Phi''(\tau)=\frac\delta{(1+\tau)^{1+\delta}}.
\end{equation}
From \cite[pp.\,360--361; cf.\
also pp.\,365--366]{Ref23} (with $W_{\mathrm{per}}^{1,p}(\Omega)$ in
place of $W_0^{1,p}(\Omega)$) we have the chain rule 
\[
 \int_0^t\big\langle k'_\eps(s),\Phi'\big(k_\eps(s)\big)\big\rangle_{W_{\mathrm{per}}^{1,r}}\dd s=\int_\Omega\Phi\big(k_\eps(x,t)\big)\dd x-\int_\Omega\Phi\big(k_{0,\eps}(x)\big)\dd x
\]
for all $t\in[\,0,T\,]$. Using $\diver\bs u_\eps =0$ we obtain 
\[
 \int_\Omega\Phi'\big(k_\eps(\cdot,t)\big)\bs{u}_\eps(t)\cdot\nabla
 k_\eps(t)\dd x =\int_\Omega \bs u_\eps(t)\cdot
 \nabla\big(\Phi(k_\eps(\cdot,t)\big) \dd x =0\quad\text{for
   a.a.\ } t\in[0,T].
\]
Inserting $z=\Phi'\big(k_\eps(\cdot,t)\big)$ into
\eqref{4.17} and using the last relation we find (recall
$\nu_0=1=\alpha_2$) 
\begin{align*}
 & \int_\Omega\Phi\big(k_\eps(x,t)\big)\dd x+ \delta\int_0^t
 \!\int_\Omega \frac{k_\eps}{\eps{+}\omega_\eps}\, 
 \frac{|\nabla k_\eps|^2}{(1{+}k_\eps)^{1+\delta}}\dd x\dd s\\
 & \quad +\,\eps\int_0^t\!\!\int_\Omega\Big(\delta\,
   \frac{|\nabla k_\eps|^r}{(1{+}k_\eps)^{1+\delta}}+ 
 k_\eps^{r-1}\big(1-\frac1{(1{+}k_\eps)^\delta}\big)\Big)\dd x\dd s\\
 & =\int_\Omega\Phi\big(k_{0,\eps}(x)\big)\dd x+ 
 \eps\int_0^t\!\!\int_\Omega\big(\kappa(s)\big)^{r-1} 
   \Big(1-\frac1{(1{+}k_\eps)^\delta}\Big)\dd x\dd s \\
 & \quad +\int_0^t\!\int_\Omega \Big(\frac{k_\eps} 
  {\eps{+}\omega_\eps {+} \eps k_\eps}\,
   \big|\bs{D}(\bs{u}_\eps)\big|^2 - k_\eps\omega_\eps\Big)
  \Big(1-\frac1{(1{+}k_\eps)^\delta}\Big)\dd x\dd s
\end{align*}
for all $t\in[\,0,T\,]$. By \eqref{4.22}, \eqref{4.26}, and
$k_\eps/\big((\eps{+} \omega_\eps)(1{+}k_\eps)\big) \geq
1/(1{+}\overline\omega(T))>0$ we find 
\begin{align}
\nonumber
 & \|k_\eps\|_{L^\infty(0,T;L^1(\Omega))} + \delta\int_Q\frac{|\nabla
   k_\eps|^2}   {(1{+}k_\eps)^\delta}\dd x\dd t + \eps\delta\int_Q 
 \frac{|\nabla k_\eps|^r}{(1{+}k_\eps)^{1+\delta}}\dd x\dd s + \eps \int_Q
 k_\eps^{r-1} \dd x \dd t \\
 \label{4.27}
 &\qquad \le
 c\Big(\frac1{1{-}\delta}+\|\bs{u}_{0,\eps}\|_{\bs{L}^2}^2+
 \|k_{0,\eps}\|_{L^1}+\|\bs{f}\|_{\bs{L}^2}^2+k_*^{r-1}\Big),  
\end{align}
where the constant $c$ is independent of $\delta$ and
$\eps$. Thus, we have estimated all the terms on the
left-hand side of \eqref{4.24} except for the second and third. 

\underline{Step 2:} To estimate $\nabla k_\eps$ we choose $p \in
{]1, 2[}$  and $\delta=(2{-}p)/p \in {]0,1[}$. 
With H\"older's inequality we find 
\begin{align*}
& \int_Q|\nabla k_\eps|^p\dd x\dd t= \int_Q\frac{|\nabla
   k_\eps|^p}{(1{+}k_\eps)^{p\delta/2}}\: (1{+}k_\eps)^{p\delta/2} \dd x\dd t 
\\
&\leq \Big( \int_Q\frac{|\nabla k_\eps|^2} 
 {(1{+}k_\eps)^\delta} \dd x \dd t \Big)^{p/2} 
   \Big(\int_Q(1{+}k_\eps)^{\delta p/(2-p)}\dd x \dd t\Big)^{(2-p)/2}
 \\
 & \leq \frac1{\delta^{p/2}}\,\Big( \delta \int_Q\frac{|\nabla k_\eps|^2} 
 {(1{+}k_\eps)^\delta} \dd x\dd t\Big)^{p/2} \:T\Big( |\Omega|{+}
 \|k_\eps\|_{L^\infty(0,T;L^1(\Omega))}\Big). 
\end{align*}
Using \eqref{4.27} this provides the estimate for the third term
on the left-hand side of \eqref{4.24}.

\underline{Step 3:} To show higher integrability of $k_\eps$ we
simply use the Gagliardo--Nirenberg interpolation from Lemma
\ref{lem:GN}  for $z\in W^{1,p}(\Omega)$ with $\Omega \subset \mathbb
R^3$ where $p\in {[1,2[}$ as in Step 2:
\[
\|z\|_{L^{4p/3}(\Omega)} \leq C_{\text{GN}} \|z\|_{L^1(\Omega)}^{1/4}
\big(\|z\|_{L^1(\Omega)}+ \|z\|_{L^p(\Omega)} \big)^{3/4},
\]
Applying this to $z=k_\eps(t)$, taking the power $4p/3$, and
integrating $t\in [0,T]$ we obtain
\begin{align*}
\int_Q |k_\eps|^{4p/3}\dd x \dd t & =  \int_0^T
\|k_\eps(t)\|_{L^{4p/3}(\Omega)}^{4p/3} \dd t 
\leq 
C_{\text{GN}}^{4p/3} \int_0^T K_\eps^{p/3} \big(
K_\eps + \|\nabla k_\eps(t)\|_{L^p(\Omega)}
\big)^p \dd t,
\end{align*}
where $K_\eps:= \| k_\eps(\cdot)\|_{L^\infty(L^1(\Omega))} \leq
C<\infty$ by Step 1. Hence, together with Step 2 the second
term on the left-hand side of \eqref{4.24}is uniformly bounded by the
right-hand side of \eqref{4.24}. 

In summary, the desired a priori estimate \eqref{4.24} is established.   
\end{proof}

\subsection{Estimates for $\{\bs{u}'_\eps, \omega'_\eps,  k'_\eps\}$} 
\label{s5.4} 

We now provide a priori estimates on the time derivative. To
obtain estimates  that
are independent of $\eps\in {]0,1]}$ we recall $r\geq 3$ and will use
$\sigma>r$ and estimate in the dual space of  $W^{1,\sigma}(\Omega)$. While
for $\bs u'_\eps$ and $\omega'_\eps$ we obtain estimates in spaces
$L^{q}\big(0,T;((W^{1,\sigma}(\Omega))^*\big)$ with $q>1$, the time derivative
$k'_\eps$ can be estimated only for $q=1$, because of the source
term $\frac{k_\eps}{\eps{+}\omega_\eps + \eps k_\eps} |\bs D(\bs
u_\eps)|^2$, for which the only $\eps$-independent a priori estimate is in
$L^1(Q)=L^1(0,T;L^1(\Omega))$. This problem will result in the
occurrence of the defect measure $\mu$. The estimates for $\bs
u'_\eps$ and $\omega'_\eps$ will work for arbitrary $r\geq 3$,
however, for the estimate of $k'_\eps$ we need to restrict $r$ to the
small interval ${[3,11/3[}$. Here the upper bound $r<11/3$ seems to be
critical for $N=3$, while $2<r<3$ might still be considered.


\begin{proposition}\label{pr:4.5} 
Let $\mathfrak D$ be fixed. \\
(A) For all $r\geq 3$ (implying $r'=r/(r{-}1)\leq 3/2$) and $\sigma>
r$ there exists a constant $C_1$ such that for all $0<\eps\le 1$ the solutions
$\{ \bs u_\eps, \omega_\eps,k_\eps\}$ of Proposition \ref{p4.1} satisfy
the estimates 
\begin{equation}\label{4.28}
\|\bs{u}'_\eps\|_{L^{r'}(0,T;(\bs{W}_{\mathrm{per},\diver}^{1,\sigma}(\Omega))^*)}
+ \|\omega'_\eps\|_{L^{r'}(0,T;(W_{\mathrm{per}}^{1,\sigma}(\Omega))^*)}
\leq C_1. 
\end{equation}
(B) For all $r\in {[3,11/3[}$ and $\sigma > 8r/(11{-}3r)$ there exists
a constant $C_2$  such that for all $0<\eps\le 1$ the solutions
$\{ \bs u_\eps, \omega_\eps,k_\eps\}$ of Proposition \ref{p4.1} satisfy
\begin{equation}\label{4.29}
\|k'_\eps\|_{L^1(0,T;(W_{\mathrm{per}}^{1,\sigma})^*)}\leq C_2. 
\end{equation}
\end{proposition}
\begin{proof} \underline{Step 1. Estimate for $\bs u'_\eps$:} For
  $\bs{w}\in\bs{W}_{\mathrm{per},\diver}^{1,\sigma}(\Omega)$, we write
  \eqref{4.15} in the form
\begin{align}
&\big\langle \bs{u}'_\eps(t) , \bs{w} 
   \big\rangle_{\bs{W}_{\mathrm{per},\diver}^{1,\sigma}}  
 =\big\langle
 \bs{u}'_\eps(t),\bs{w}\big\rangle_{\bs{W}_{\mathrm{per},\diver}^{1,r}}
 \nonumber\\
\label{4.30}
& =\int_\Omega\!\big(\bs{u}_\eps(t) {\otimes}
  \bs{u}_\eps(t)\big) {:} \nabla\bs{w}\dd
  x-\nu_0\int_\Omega\!\frac{k_\eps(t)}{\eps\!+\!\omega_\eps(t)}\, 
 \bs{D}\big(\bs{u}_\eps(t)\big) {:} \bs{D}(\bs{w})\dd x
\\ 
& \quad-
\eps\int_\Omega\! \Big(\big|\bs{D}\big(\bs{u}_\eps(t)\big)\big|^{r-2} 
 \bs{D}\big(\bs{u}_\eps(t)\big) {:} \bs{D}(\bs{w})+ 
  \big|\bs{u}_\eps(t)\big|^{r-2}\bs{u}_\eps(t) {\cdot} \bs{w}\Big)\dd x
+\int_\Omega\bs{f}(t)\cdot\bs{w}\dd x \nonumber\\
& = \sum\limits_{m=1}^4 I_{\eps,m}(t) \quad \text{ for a.a.\ } t\in[0,T].
 \nonumber
\end{align}
The aim is to show $|I_{\eps,m}(t)|\leq f_{\eps,m}(t)\|\bs
w\|_{\bs W^{1,\sigma}(\Omega)} $ with $f_{\eps,m}$ bounded in
$L^{\overline q_m}(0,T)$ for some $\overline q_m \geq r/(r{-}1)$. For
this, we proceed as in Remark \ref{r3.1}, but use now that $\bs w\in
\bs W^{1,\sigma}_\text{per,div}(\Omega)$ is fixed.

For $I_{\eps,1}$ we use $\nabla \bs
w\in \bs L^\sigma(\Omega)$ and need to bound $|\bs u_\eps {\otimes} \bs u_\eps|
\leq |\bs u_\eps|^2$ in $L^{\sigma'}(\Omega)$, which means $\bs u_\eps\in \bs
L^p(\Omega)$ with $p=2\sigma/(\sigma{-}1)$. For this we use the bounds
\eqref{4.22} for $\bs u_\eps$, which allow us to apply Lemma
\ref{lem:GN}(B)  with $(s_1,p_1)=(\infty,2)$, $(s_2,p_2)=(2,2)$, $N=3$,
and $\theta=3/(2\sigma)<1/2$. 
This provides the desired $p=2\sigma/(\sigma{-}1)$ and $\overline q_1 
= s=4\sigma/3$. 

To estimate $I_{\eps,2}$ we use $\eps+\omega_\eps(x,t)\geq \underline
\omega(T)>0$ and need to bound 
\[
|k_\eps \bs D(\bs u_\eps)| = k_\eps^{1/2}\; |k_\eps^{1/2}\bs D(\bs
u_\eps)| \quad \text{ in } L^{\overline q_2}(0,T;L^{\sigma'}(\Omega)).
\]
By \eqref{4.22} we have a uniform bound for $|k_\eps^{1/2}\bs D(\bs
u_\eps)|$ in $L^2(Q)=L^2(0,T;L^2(\Omega))$. Moreover, \eqref{4.24}
provides uniform bounds for $\|k_\eps\|_{ L^\infty(0,T;L^1(\Omega) ) }$
and for $\|\nabla k_\eps\|_{\bs L^p(Q)}$ with $p\in {[1,2[}$. Hence,
restricting to $\overline q_2\in [1,2]$ we proceed as follows:
\begin{align*}
&\|k_\eps \bs D(\bs u_\eps)\|_{L^{\overline
    q_2}(0,T;L^{\sigma'}(\Omega))}^{\overline q_2} \leq \int_0^T \big(
\| k^{1/2}_\eps\|_{L^{2\sigma/(\sigma-2)}} \|k_\eps^{1/2}\bs D(\bs
u_\eps)\|_{L^2}\big)^{\overline q_2} \dd t 
\\ &
\leq \int_0^T \!\!\| k_\eps\|_{L^{\sigma/(\sigma-2)}}^{\overline q_2/2} 
 \|k_\eps^{1/2}\bs D(\bs u_\eps)\|_{L^2}^{\overline q_2} \dd t  
\overset{\text{H\"older}}\leq \Big( \int_0^T \!\! \|
k_\eps\|_{L^{\sigma/(\sigma-2)}}^{ \overline q_2/(2-\overline q_2)}
\dd t \Big)^{(2-\overline q_2)/2} \Big( \int_Q \!k_\eps|\bs D(\bs
u_\eps)|^2\dd t   \Big)^{\overline q_2/2}  .
\end{align*}
The second term in the last product is already uniformly bounded. To
estimate the first term we apply Lemma \ref{lem:GN}(B)
with $(s_1,p_1)=(\infty,1)$, $s_2=p_2 \in {[1,2[}$, $N=3$, and
$\theta=6 p_2/((4p_2{-}3)\sigma) \in {]0,1[}$, where we use
$\sigma>r\geq 3$ such that $p_2$ can be chosen close to $2$. From the
interpolation condition \eqref{eq:Interp.s} we obtain the range of
possible $\overline q_2$ via
\[
\frac2{\overline q_2}-1= \frac{2{-}\overline q_2}{\overline q_2} =
\frac1s = (1{-}\theta)\frac1{s_1} + \theta \frac1{s_2}= 0 +
\theta\frac1{p_2} = \frac{6}{(4p_2{-}3)\sigma}. 
\] 
Thus, we are able to choose all $\overline q_2 \in
{[1,10\sigma/(5\sigma{+}6)[}$ by adjusting $p_2$ suitably. As
$\sigma>r\geq 3$ we see that $\overline q_2=3/2$ is always
admissible.

Using $\sigma\geq r\geq 3$ and H\"older's inequality, we obtain
\[
\big|I_{\eps,3}(t)\big|\le f_{\eps,3}(t)
\|\bs{w}\|_{\bs{W}^{1,\sigma}}\quad\text{ with }
f_{\eps,3}(t)=C\eps\big\|\bs{u}_\eps(t)\big\|_{\bs{W}^{1,r}}^{r-1}. 
\]
By the uniform bound \eqref{4.22} we obtain
$\|f_{\eps,3}\|_{L^{r'}(0,T)} \leq C_* \eps^{1/(r{-}1)}$ with a constant
$C_*$ independent of $\eps$. Thus, we can choose $\overline
q_3= r'=r/(r{-}1)\leq 3/2$.

With $|I_{\eps,4}(t)| \leq \|\bs f(t)\|_{L^2} \|\bs w(t)\|_{\bs
  L^2}\leq C \|\bs f(t)\|_{L^2}\| \bs w\|_{\bs W^{1,\sigma}}$ and $\bs
  f \in \bs L^2(Q)=L^2(0,T;\bs L^2(\Omega))$ we obtain $\overline
  q_4=2$, and conclude that in all cases we have $\overline
  q_m\geq r'=r/(r{-}1)$ and the first part of \eqref{4.28} is
  established. 
\medskip

\underline{Step 2. Estimate for $\omega'_\eps$:} We proceed as in Step
1 by writing \eqref{4.16} in the form 
\[
\big\langle \omega'_\eps(t),\psi \big\rangle_{W^{1,\sigma}} =
\sum_{m=1}^5 J_{\eps,m}(t) \quad \text{ with } |J_{\eps,m}(t)| \leq
g_{\eps,m}(t)\| \psi\|_{W^{1,\sigma}},
\]
where $g_{\eps,m}$ has to be bounded in $L^{\widetilde q_m}(0,T)$ for
suitable $\widetilde q_m\geq r'=r/(r{-}1)$. Exploiting Lemma
\ref{l4.2}, namely $0< \underline \omega(T)\leq \omega_\eps(x,t)\leq
\overline\omega(0)=\omega^*$ and \eqref{4.23} and proceeding as in
Step 1 we easily find $\widetilde q_1=\widetilde q_3 = \widetilde
q_5=\infty$, $\widetilde q_2 = 10\sigma/(5\sigma{+}6) \geq 3/2$, and
$\widetilde q_4=r'\leq 3/2$. Thus, the second part of \eqref{4.28},
and hence all of \eqref{4.28}, is established.  \medskip

\underline{Step 3. Estimate for $k'_\eps$:} We again write 
\begin{align}
 \big\langle k'_\eps(t),z\big\rangle
\label{4.31}
& =-\int_\Omega z\bs{u}_\eps(t)\cdot\nabla k_\eps(t)\dd x-\int_\Omega\frac{k_\eps(t)}{\eps+\omega_\eps(t)}\,\nabla k_\eps(t)\cdot\nabla z\dd x\\
& \quad +\nu_0\!\int_\Omega\frac{k_\eps(t)}{\eps+\omega_\eps(t)+\eps k_\eps(t)}\,\big|\bs{D}\big(\bs{u}_\eps(t)\big)\big|^2z\dd x-\alpha_2\int_\Omega k_\eps(t)\omega_\eps(t)z\dd x\nonumber\\
& \quad- \eps\!\int_\Omega\big(\big|\nabla k_\eps(t)\big|^{r-2}\nabla
k_\eps(t)\cdot\nabla z+\big|k_\eps(t)\big|^{r-2}k_\eps(t)z\big)\dd x
 +\,\eps\big(\kappa(t)\big)^{r-1}\int_\Omega z\dd x\nonumber\\
& =:\sum\limits_{m=1}^7 K_{\eps,m}(t)\nonumber
\end{align}
and have to show that $K_{\eps,m}(t)\leq
h_{\eps,m}(t)\|z\|_{W^{1,\sigma}}$, where $h_{\eps,m}$ is bounded in
$L^1(0,T)$ independently of $\eps\in {]0,1[}$ and $m=1,\ldots,7$. 

Before starting the estimates we note that the condition $r\in
{[3,11/3[}$ and $\sigma> 8r /(11{-}3r)$ implies $\sigma >12$, which
will be useful below.  

For $m=1$ we integrate by parts using $\diver \bs u_\eps=0$  and obtain
\[
|K_{\eps,1}(t)| = \Big|\int_\Omega k_\eps \bs u_\eps {\cdot} \nabla
z\dd x \Big| \leq h_{\eps,1}(t) \|z\|_{W^{1,\sigma}} \quad \text{with
} h_{\eps,1}(t)=\|k_\eps \bs u_\eps\|_{L^{\sigma'}}.
\]
Using \eqref{4.22} for $\bs u_\eps$ and applying Lemma \ref{lem:GN}
with $(s_1,p_1)=(\infty,2)$, $(s_2,p_2)=(2,2)$, $N=3$, and
$\theta=3/5$ we find $(s,p)=(10/3,10/3)$ which means that $\bs u_\eps$
is uniformly bounded in $L^{10/3}(Q)$. Using the uniform bound
\eqref{4.24} for $k_\eps$ in $L^q(Q)$ for all $q \in {[1,8/3[}$ 
we can choose $q$ such that $\frac1q+\frac3{10}\leq 1/{\sigma'}<1$ 
as $\sigma>40/13$ and obtain 
\[ 
\int_0^T h_{\eps,1}(t) \dd t \leq \int_0^T C\| k_\eps(t)\|_{L^{q}(\Omega)} 
\| \bs u_\eps(t) \|_{L^{10/3}(\Omega)} \dd t \leq C_T \|
k_\eps\|_{L^q(Q)} \| \bs u_\eps\|_{L^{10/3}(Q)} \leq C_{T,1}. 
\]

For $m=2$ we again use \eqref{4.24} and $\sigma>8$. Choosing $p\in
{[1,2[}$ with $3/(4p) + 1/p +1/\sigma \leq 1$ H\"older's inequality
gives 
\[
\int_0^T|K_{\eps,2}(t)|\dd t \leq \int_0^T \|k_\eps\|_{L^{4p/3}} 
 \|\nabla k_\eps\|_{L^p} \|\nabla z\|_{L^\sigma}\dd t \leq C_{T,2}\|
k_\eps\|_{L^{4p/3}(Q)} \|\nabla k_\eps\|_{L^p(Q)} \|
z\|_{W^{1,\sigma}}. 
\] 

The case $m=3$ follows easily as $\|z\|_{L^\infty(\Omega)} \leq C \|
z\|_{W^{1,\sigma} }$ because $\sigma>N$. Together with the simple
energy estimate \eqref{4.22} (uniform boundedness of the dissipation)
we obtain 
\[
\int_0^T |K_{\eps,3}(t)|\dd t \leq C \int_Q
\frac{k_\eps}{\eps{+}\omega_\eps} |\bs D(\bs u_\eps)|^2 \dd x \dd t
\|z\|_{L^\infty}  \leq C_{3}\| z\|_{W^{1,\sigma}}.
\]
The case $m=4$ is also trivial, since $|K_{\eps,4}(t)| \leq C\|
k_\eps(t)\| \omega^* \|z\|_{L^\infty}$. 

The most difficult term is $K_{\eps,5}$ because we do not have an a
priori bound on $\eps|\nabla k_\eps|^r$. We adapt the method developed
in Step 2 of the proof of Proposition \ref{pr:4.4}. Using 
\[
|K_{\eps,5}(t)| \leq h_{\eps,5}(t)\|z\|_{W^{1,\sigma}} \quad 
\text{ with } h_{\eps,5}(t) = \eps \big\| |\nabla k_\eps(t)|^{r-1} 
\big\|_{L^{\sigma'}}
\]
we proceed as follows:
\begin{align*}
\int_0^T h_{\eps,5}\dd t & = \eps \int_0^T \|\nabla
k_\eps(t)\|_{L^(r-1)\sigma'}^{r-1} \dd t \leq \eps  T^{1/\sigma}
\|\nabla k_\eps\|_{L^{(r-1)\sigma'}(Q)}^{r-1}\\
& \leq \eps T^{1/\sigma} \Big( \int_Q \frac{|\nabla
  k_\eps|^{(r-1)\sigma'}}{(1{+}k_\eps)^\rho}  (1{+}k_\eps)^\rho \dd x
\dd t  \Big)^{1/\sigma'}\\
\intertext{for a $\rho>0$ to be chosen appropriately. Applying
  H\"older's inequality with $p= r'/\sigma'>1$ and using $\eps =
  \eps^{1/r} \eps^{1/(p\sigma')}$ we continue}
&\leq \eps^{1/r} T^{1/\sigma} \Big( \int_Q \frac{\eps |\nabla k_\eps|^r} 
  {(1{+} k_\eps)^{p\rho}} \dd x \dd t \Big)^{1/(p\sigma')} 
  \Big( \int_Q (1{+} k_\eps)^{p'\rho} \dd x \dd t \Big)^{1/(p'\sigma')}.
\end{align*}
According to \eqref{4.24} both integral terms are uniformly bounded if
we can choose $\rho$ such that $p\rho \in {]1,2]}$ and $p'\rho < 8/3$.
Writing $\kappa=1/p$ this means $\kappa < \rho < \min\{ 2\kappa,
8(1{-}\kappa)/3 \}$, which has solutions $\rho$ if and only if $\kappa
\in {]0,8/11[}$, i.e.\ we need $p=r'/\sigma'>11/8$ which in term can
only be possible if $r'>11/8$ or $r<11/3$. Then, $p=r'/\sigma'>11/8$
is equivalent to $\sigma> 8r/(11{-}3r)$. This
explains the restriction for $r$ and $\sigma$ in \eqref{4.29} and
provides the $L^1$ bound $ \int_0^T |K_{\eps,5}(t)| \dd t \leq
\eps^{1/r} C_{r,\sigma} \|z\|_{W^{1,\sigma}} $. 

The estimate of $K_{\eps,6}$ follows easily from \eqref{4.24} using
$r{-}1\in {[2,8/3[}$, which implies $\|k_\eps\|_{L^{r-1}(Q)}
\leq C$ and thus 
\[
\int_0^T |K_{\eps,6}(t)|\dd t \leq \int_0^T \eps
\|k_\eps\|_{L^{r-1}}^{r-1} \dd t\: \| z\|_{L^\infty} \leq  \eps C
\|z\|_{W^{1,\sigma}}. 
\]

The case of $K_{\eps,7}$ is trivial. 

For later use in the limit passage $\eps\to 0$ we note that   
\begin{equation}
  \label{eq:K567}
   \int_0^T \!\! \big( |K_{\eps,5}(t)|+|K_{\eps,6}(t)|+|K_{\eps,7}(t)|\big)
   \dd t \ \leq \ \eps^{1/r} C_{r,\sigma} \|z\|_{W^{1,\sigma}}.
\end{equation}
Hence, the a priori estimate \eqref{4.29}
for $k'_\eps$ is established. 
\end{proof}

\subsection{Convergent subsequences}
\label{s5.5}

After having derived a series of a priori estimates we are now
able to choose weakly converging subsequences for $\eps\to 0$. Of
course the major step is to identify the limits of the nonlinear
terms. For simplicity we now choose one fixed $r_*\in {[3,11/3[}$ and
a $\sigma_*>12$, which implies that Part (A) and (B) of Proposition
\ref{pr:4.5} can be applied. From \eqref{4.18},
\eqref{eq:AppSolEE}, \eqref{4.24}, \eqref{4.28}, and \eqref{4.29}
we obtain a limit triple $ \{\bs u, \omega, k\}$ with the
properties 
\begin{equation}
  \label{eq:LimitTriple}
\left.\begin{aligned}
&\underline \omega \leq \omega \leq \overline \omega \text{ a.e.\ on }Q,  
\\
&\bs u \in L^2(0,T;\bs
  W^{1,2}(\Omega))\cap L^\infty(0,T;\bs L^2(\Omega)) \cap
  W^{1,r'_*}\big(0,T;(\bs W^{1,\sigma_*}_\text{per,div}(\Omega))^*\big), 
\\
& \omega \in L^\infty(Q)\cap L^2(0,T;W^{1,2}(\Omega)) \cap
  W^{1,r'_*}\big(0,T;(W^{1,\sigma_*}_\text{per}(\Omega))^*\big),
\\ 
& k\in  L^\infty(0,T;L^1(\Omega)) \cap L^{4p/3}(Q) \cap
  L^p(0,T;W^{1,p}_\text{per}(\Omega)) \cap
  \text{BV}(0,T;(W^{1,\sigma_*}_\text{per}(\Omega))^*\big)
\end{aligned}
\right\} 
\end{equation}
for all $p\in {[1,2[}$, such that along a suitable subsequence (not
relabeled) we have 
\begin{subequations}
\label{eq:LinCvg}
\begin{align}
& \label{eq:LinCvg.a}
   \bs{u}_\eps \weak \bs{u}\ \text{ in } L^2\big(0,T; 
  \bs{W}_{\mathrm{per},\diver}^{1,2}(\Omega)\big) \  \text{ and 
 weakly$^*$ in } L^\infty\big(0,T;\bs{L}^2(\Omega)\big),
\\
& \label{eq:LinCvg.b}
 \bs{u}'_\eps \weak \bs{u}'\ \text{ in } 
 L^{r'_*}\big(0,T;(W_{\mathrm{per},\diver}^{1,\sigma_*}(\Omega))^*\big),  
\\[2mm] 
& \label{eq:LinCvg.c}
 \omega_\eps \weak \omega \ \text{ in }  L^2\big(0,T; 
  W_{\mathrm{per}}^{1,2}(\Omega)\big) \ \text{ and weakly$^*$ in } 
 L^\infty(Q),
 \\ 
& \label{eq:LinCvg.d}
\omega'_\eps \weak \omega' \ \text{  in } L^{r'_*}\big(0,T; 
                  (W_{\mathrm{per}}^{1,\sigma_*}(\Omega))^*\big),
\\[2mm] 
& \label{eq:LinCvg.e}
        k_\eps \weak  k\ \text{  in } L^p\big(0,T;W^{1,p}_\text{per}(\Omega)\big)
        \text{  and in }
         L^{4p/3}(Q) \ \text{ for all } p \in {[1,2[}. 
\end{align}
\end{subequations}
These weak convergences imply the corresponding properties of the
limits $ \bs u$ and $ \omega$ in \eqref{eq:LimitTriple}. 
Moreover, $\|k\|_{L^\infty(0,T;L^1(\Omega))}\leq C< \infty$
follows from \eqref{4.24} and \eqref{eq:LinCvg.e} by a routine
argument. As in \cite[Sec.\,1.3.2]{BarPre12} the space $BV(0,T;X)$,
where $X$ is a Banach space, denotes all functions $g:[0,T]\to X$ such
that $\text{Var}_X(g,[a,b]):=\sup \sum_{i=1}^N \|g(t_i){-}g(t_{i-1})\|_X
< \infty$ where the supremum is taken over all finite partitions $a\leq
t_0<t_1<\cdots <t_N\leq b$. Clearly, \eqref{4.29} implies 
$\text{Var}_{(W^{1,\sigma}_\text{per})^*}(k_\eps,[0,T])=\|k'_\eps\|_{L^1(0,T;
  (W^{1,\sigma}_\text{per})^*)}\leq C_2$. Since for all
partitions we have 
\[
\sum_{i=1}^N\|k(t_i){-}k(t_{i-1})\|_{(W^{1,\sigma}_\text{per})^*} \leq
\liminf_{\eps\to 0}
\sum_{i=1}^N\|k_\eps(t_i){-}k_\eps(t_{i-1})\|_{(W^{1,\sigma}_\text{per})^*}
\leq C_2,
\]
which provides $\|k\|_{\text{BV}(0,T;(W^{1,\sigma_*}_\text{per}(\Omega))^*)}
\leq C_2 < \infty$  as stated at the end of \eqref{eq:LimitTriple}.    

We next apply the Aubin-Lions-Simon lemma (see 
\cite[Cor.\,4, p.\,85]{Ref24}, 
\cite[Th.\,5.1, p.\,58]{Ref18}, or \cite[Lem.\,7.7]{Roub13NPDE}) to
obtain strong convergence. By taking a further subsequence (not
relabeled) Vitali's theorem implies the pointwise convergence almost
everywhere. 
\begin{subequations} 
\label{eq:AubinLions}
\begin{align}
\label{eq:AubinLions.u}
&\bs{u}_\eps\to \bs{u}\ \text{  in } \bs{L}^s(Q) \ \text{  for all $s
  \in {[1,10/3[} $ \ and  a.e.\ in } Q,
\\
\label{eq:AubinLions.o}
&\omega_\eps\to \omega\,\ \text{  in } {L}^p(Q) \text{ for all $p>1$
  \ and  a.e.\ in }Q,
\\
\label{eq:AubinLions.k}
&k_\eps\, \to \, k \ \text{  in } {L}^q(Q) \text{ for all } q \in {[1,8/3[}
  \ \text{ and  a.e.\ in }Q,
\end{align}
\end{subequations}
To obtain the results in \eqref{eq:AubinLions.o} and
\eqref{eq:AubinLions.k} we first derive strong convergence
for $s=p=q=2$ and then use the boundedness of the sequence for higher
$s$, $p$, and $q$ to obtain strong convergence for intermediate values
by Riesz interpolation (use \eqref{eq:bs.u.10/3} for $\bs u_\eps$).

We are now ready to consider also the limits of the nonlinear
terms. We first treat the diffusive terms.

\begin{lemma}\label{lem:DiffNonlin} Along the chosen subsequences for
  $\eps \to 0$  we
  have the convergences
\begin{subequations} 
\label{eq:DiffCvg}
\begin{align}
\label{eq:DiffCvg.a}
 \frac{k_\eps}{\eps{+}\omega_\eps}\bs{D}(\bs{u}_\eps)
\weak  \frac k\omega \bs D(\bs u)   \text{ and } 
 \frac{k_\eps}{\eps{+}\omega_\eps}\nabla \omega_\eps
\weak  \frac k\omega \nabla \omega \ &\text{in }\bs L^s(Q) \text{ for
  all } s \in {[1,16/11[}, 
\\
\label{eq:DiffCvg.b}
 \frac{k_\eps}{\eps{+}\omega_\eps}\nabla k_\eps
\weak  \frac k\omega \nabla k \hspace{6em} &\text{in } \bs L^\sigma(Q)\text{ for
  all } \sigma \in {[1,8/7[}.
\end{align} 
\end{subequations}
\end{lemma}
\begin{proof} We first recall the weak convergences of the gradients
  $\bs D(\bs u_\eps)$, $\nabla \omega_\eps$, and $\nabla k_\eps$ in
$L^p(Q)$ for all $p\in {[1,2[}$, see \eqref{eq:LinCvg}. 
Next we establish the strong convergence
\begin{equation}
  \label{eq:sqrt.Cvg}
  \Big(\frac{k_\eps}{\eps{+}\omega_\eps}\Big)^{1/2} \ \to \ \Big(\frac k\omega
  \Big)^{1/2} \quad \text{ in }L^q(Q) \text{ for all } q\in {[1,16/3[}.
\end{equation}
To see this we use the explicit estimate  
\begin{align*}
\Big\|\big(\frac{k_\eps}{\eps{+}\omega_\eps}\big)^{1/2}- 
   \big(\frac k\omega\big)^{1/2}\Big\|_{L^q(Q)} \!
 &\leq 
  \Big\|\big(\frac{k_\eps}{\eps{+}\omega_\eps}\big)^{1/2} \! -
     \big(\frac{k}{\eps{+}\omega_\eps}\big)^{1/2}\Big\|_{L^q(Q)} 
  \! +\Big\|\big( \frac{k}{\eps{+}\omega_\eps} \big)^{1/2} 
        -  \big(\frac k\omega\big)^{1/2}\Big\|_{L^q(Q)} \\
&\leq \frac{\| k_\eps {-} k\|_{L^{q/2}(Q)}^{1/2} }
           {(1{+}\underline\omega(T))^{1/2}} 
   +
 \frac{\big\|(\eps{+}\omega_\eps -\omega)\,k^{1/2}
 \big\|_{L^q(Q)}} {2(1{+}\underline\omega(T))^{3/2}}  .
\end{align*}
Clearly, the first term on the right-hand side tends to $0$ using
\eqref{eq:AubinLions.k} and $q/2<8/3$. For the second term we can
still choose $\widetilde q\in {]q,16/3[}$ and $\widetilde p \gg 1$
such that $1/q=1/\widetilde q + 1/\widetilde p$. Then, H\"older's
inequality, $k^{1/2}\in L^{\widetilde q}(Q)$, and
\eqref{eq:AubinLions.o} for $p=\widetilde p$ yield the convergence to
$0$. Hence, the convergence \eqref{eq:sqrt.Cvg} is established.

Now using the weak convergences $\bs D(\bs u_\eps) \weak \bs D(\bs
u)$ and  $\nabla \omega_\eps \weak \nabla \omega$, and $\nabla k_\eps
\weak \nabla k$ in $L^p(Q)$ for $p\in {[1,2[}$ and \eqref{eq:sqrt.Cvg}
we obtain the weak convergences 
\[
 \big(\tfrac{k_\eps}{\eps{+}\omega_\eps}\big)^{1/2}\bs{D}(\bs{u}_\eps)
\weak  \big(\tfrac k\omega\big)^{1/2} \bs D(\bs u), \quad 
 \big(\tfrac{k_\eps}{\eps{+}\omega_\eps}\big)^{1/2}\nabla \omega_\eps
\weak  \big(\tfrac k\omega\big)^{1/2} \nabla \omega,\quad
 \big(\tfrac{k_\eps}{\eps{+}\omega_\eps}\big)^{1/2}\nabla k_\eps
\weak  \big(\tfrac k\omega\big)^{1/2} \nabla k 
\] 
in $L^q(Q)$ for all $q \in {[1,16/11[}$. 

However, by the standard a priori estimates \eqref{eq:AppSolEE} we see
that the first two sequences are bounded in $L^2(Q)$ and hence 
converge weakly in $L^2(Q)$ as well. The convergence of the third term cannot be
improved, because we don't have appropriate a priori bounds. 

Multiplying once again by
$\big(k_\eps/(\eps{+}\omega_\eps)\big)^{1/2}$, which converges
strongly according to \eqref{eq:sqrt.Cvg}, we obtain the results in
\eqref{eq:DiffCvg}. 
\end{proof}

\subsection{Limit passage $\eps\to 0$ and appearance of
  the defect measure}\label{s5.6}

In this subsection we finalize the proof of Theorem
\ref{th:MainExist}.

Using the convergences derived above it is now straight forward
to perform  the limit passage $\eps\to 0$ in the equation for $\bs u_\eps$
and $\omega_\eps$. In the energy equation for $k_\eps$ we have to be a
little more careful to show the occurrence of the defect measure
$\mu$. 

In the Steps 1 to 3 the limit $\eps\to 0$ will be done with
test functions with high integrability $\rr$ in $t\in 
[0,T]$ taking values in the Sobolev $W^{1,\sss}(\Omega)$ with large
$\sss$. This choice will be independent of the chosen $r_*$
in the regularization terms. After the artificial $r_*$ has
disappeared in the limit, in Step 4 we discuss which minimal $\rr$ and
$\sss$ can be chosen in the weak form. \medskip 

\noindent
\underline{Step 1. Limit in the momentum balance for $\bs u_\eps$, from
  \eqref{4.11} to \eqref{3.7}:}  
We consider a fixed test function $\bs v \in
L^\rr\big(0,T;\bs W^{1,\sss}_\text{per,div}(\Omega))^*\big)$ with
$\rr=4$ and $\sss\geq s_*>12$ and discuss the convergence of the five terms
on the left-hand side of \eqref{4.11} individually. 

The first term is linear in $\bs u'_\eps$ and converges because of
\eqref{eq:LinCvg.b}.  The second term can be rewritten as 
$\int_\Omega \big(\bs u_\eps {\otimes} \bs u_\eps) : \nabla \bs v \dd x
\dd t $ and converges by \eqref{eq:AubinLions.u}. 

For the third term we use the nonlinear convergences from Lemma
\ref{lem:DiffNonlin}, cf.\ the first in \eqref{eq:DiffCvg.a}. 
The fourth and fifth terms converge to $0$ by the estimate 
$\int_0^T|I_{\eps,3}(t)|\dd t \leq C_*\eps^{1/(r_*{-}1)} \|\bs D(\bs
v)\|_{L^{r_*}(\bs L^{\sigma_*})} {\leq} C\eps^{1/(r_*{-}1)} \| \bs
v\|_{L^\rr(\bs W^{1,\sss})}$, see Step\,1 of the proof of Proposition\,\ref{pr:4.5}. 

Thus, \eqref{3.7} is established for test functions $v\in
L^\rr\big(0,T;\bs W^{1,\sss}_\text{per,div}(\Omega))^*\big)$. 
\medskip

\noindent
\underline{Step 2. Limit for $\omega_\eps$, from \eqref{4.12} to \eqref{3.8}:}  
This case works similar as Step 1. \medskip

\noindent
\underline{Step 3. Limit in the energy equation for $k_\eps$, from
  \eqref{4.13} to \eqref{2.6}:} 
For this limit passage we choose a test function $z\in
C_{\mathrm{per},T}^1(\overline{Q})$,\vspace{0.1em} because we want to take the limit
of the dissipation which is bounded only in $L^1(Q)$. 

The first term of the left-hand side in\eqref{4.13} is integrated by
parts in time to obtain
\begin{align*}
 \int_0^T \!\!\big\langle
 k'_\eps(t),z(t)\big\rangle_{W_{\mathrm{per} }^{1,r}}\dd t =
 \int_\Omega k_{0,\eps} z(\cdot,0)\dd x - \int_Q k_\eps z'\dd x \dd t \
 \to \  \int_\Omega k_0 z(\cdot,0)\dd x - \int_Q k z'\dd x \dd t 
\end{align*}
by \eqref{4.8} and and \eqref{eq:LinCvg.e}. 
For the second term we use \eqref{eq:AubinLions} and conclude
\begin{align*}
 \int_Q z \bs u_\eps {\cdot} \nabla k_\eps \dd x \dd t = - \int_Q
 k_\eps \nabla \bs u_\eps{\cdot} \nabla z\dd x \dd t 
\ \to \ - \int_Q k\bs u{\cdot} \nabla z\dd x \dd t  .
\end{align*}

For the third term Lemma \ref{lem:DiffNonlin} can be exploited (cf.\
\eqref{eq:DiffCvg.a}) to find
\[
\int_Q\frac{k_\eps}{\eps+\omega_\eps}\,\nabla k_\eps\cdot\nabla z\dd
x\dd t \ \to \ \int_Q\frac k\omega\,\nabla k\cdot\nabla z\dd x\dd t. 
\]

We return to the fourth term at the end and continue with the fifth
term. Using \eqref{eq:AubinLions} and $\omega_\eps^+ = \omega_\eps
\geq \underline \omega(\cdot) >0$ we easily find $\int_Q k_\eps
\omega_\eps^+ z \dd x \dd t \to \int_Q k\omega z \dd x \dd t $.

The sixth and seventh term on the left-hand side and the single term
on the right-hand side converge to $0$, which was establish in
Step 3 of the proof of Proposition \ref{pr:4.5}, see \eqref{eq:K567}.

For the fourth term, it 
remains to prove the \emph{appearance of the non-negative defect measure}
$\mu\in \mathcal{M}_\geq(\overline{Q})$ such that
\begin{align}
 \label{4.41}
\int_Q\frac{\nu_0 k_\eps}{\eps{+}\omega_\eps{+}\eps
  k_\eps}\,\big|\bs{D}(\bs{u}_\eps)\big|^2\phi\dd x\dd t \
\longrightarrow \ \int_Q\frac{\nu_0 k}\omega \, 
 \big|\bs{D}(\bs{u})\big|^2\phi\dd x \dd t  +
   \int_{\overline{Q}}\phi\dd \mu \ 
\text{ for all } \phi\in C(\overline{Q}) .       
\end{align}
Indeed, by the positivity of the integrand and the a priori
estimate \eqref{4.22} we can apply Riesz' Representation Theorem
for linear continuous functionals on $C(\overline{Q})$. Hence, there
exist  $\widehat{\mu} \in \mathcal{M}_\geq (\overline{Q}) 
$  such that
\[
\int_Q\frac{\nu_0 k_\eps}{\eps{+}\omega_\eps{+}\eps k_\eps} 
\,\big|\bs{D}(\bs{u}_\eps)\big|^2\phi\dd x\dd t \ \to \ 
 \int_{\overline{Q}}\phi\dd\widehat{\mu} \ 
\text{  for all } \phi\in C(\overline{Q}).
\]

As in Lemma \ref{lem:DiffNonlin} we can show that
$\big(\frac{k_\eps}{\eps{+}\omega_\eps{+}\eps k_\eps} \big)^{1/2} \bs
D(\bs u_\eps)$ converges weakly to $(k/\omega)^{1/2}\bs D(\bs u)$ in
$\bs L^2(Q)$. Of course, this weak convergence remains true if we
multiply by a continuous function $\psi\in C(\overline Q)$. 
Thus, the lower semi-continuity of the $L^2$ norm yields 
\begin{align*}
 \int_Q\psi^2\dd\widehat{\mu}= \lim\limits_{\eps\to0}
 \int_Q\frac{\nu_0 k_\eps}{\eps+\omega_\eps+\eps  k_\eps}\,
  \big|\bs{D}(\bs{u}_\eps)\big|^2\psi^2\dd x\dd t 
 \geq  \int_Q\frac{\nu_0k}\omega\,\big|\bs{D}(\bs{u})\big|^2\psi^2\dd x\dd t
\end{align*}
for all $\psi\in C(\overline{Q})$. Thus, the linear
functional
$\phi  \mapsto  \int_Q \phi \dd \widehat\mu  - \int_Q\frac{\nu_0k}
\omega\,\big|\bs{D}(\bs{u})\big|^2\phi\dd x\dd t$
is non-negative and defines the desired defect measure $\mu \in
\mathcal M_\geq (\overline Q)$, and 
\[
 \int_{\overline{Q}}\phi \dd\widehat{\mu}=\int_Q\frac
{\nu_0k}\omega\,\big|\bs{D}(\bs{u})\big|^2\phi\dd x\dd t + 
 \int_{\overline{Q}}\phi \dd{\mu} 
\quad\text{for all }  \phi\in C(\overline{Q}),
\]
 which gives the desired convergence \eqref{4.41}.
\medskip

\noindent \underline{Step 4. More test functions:} 
After having passed to the limit $\eps\to 0$ the regularization terms
involving the exponent $r$ have disappeared. From the a priori
estimates \eqref{eq:LimitTriple} for $\{\bs u,\omega,k\}$ we know that 
$\bs u{\otimes} \bs u\in L^{5/3}(Q)$ and $\frac k\omega\bs D(\bs u)\in
L^q(Q)$ for all $q\in {[1,16/11[}$. Thus, by density we 
can extend the set of test function $\bs v$ in
\eqref{3.6} can be chosen in
$L^\rr\big(0,T;W^{1,\sss}_\text{per,div}(\Omega)\big)$ for any
$\rr>16/5$ and $\sss>16/5$. This proves \eqref{3.7} and \eqref{3.8} for
the full set of test functions. 

Moreover, we find $\bs u' \in
L^q\big((W^{1,q'}_\text{per,div}(\Omega))^*\big)$ for all $q\in
{[1,16/11[}$, which proves \eqref{3.6}. 
\medskip

\noindent \underline{Step 5. Several further statements:} 
To derive \eqref{3.5} we define the functional 
$ 
\mathcal J:( k,\bs u,\omega) \mapsto \int_Q k\,(|\bs D(\bs u)|^2{+}|\nabla
\omega|^2)\dd x \dd t  
$ 
and use the a priori estimate $\mathcal J(k_\eps,\bs
u_\eps,\omega_\eps) \leq C$, which follows from \eqref{eq:AppSolEE}
since $\omega_\eps\geq \underline\omega(T)>0$. The functional is
convex in $\bs u$ and $\omega$, hence it is lower semicontinuous with
respect to strong convergence in $k$ (see \eqref{eq:AubinLions.k}) and
weak convergence for $(\bs u,\omega)$ (see \eqref{eq:LinCvg.a} and
\eqref{eq:LinCvg.c}), so that 
\[
\mathcal J(k,\bs u,\omega) \leq \liminf_{\eps \to 0} \mathcal
J(k_\eps,\bs
u_\eps,\omega_\eps) \leq C,
\]
which is the desired estimate \eqref{3.5}. The limit passage $\eps\to 0$ in
the pointwise a priori estimates \eqref{4.18} leads immediately to the
pointwise estimates \eqref{3.2} for $\omega$ and $k$.

By \eqref{eq:LinCvg.b} and \eqref{eq:LinCvg.d} the functions $\bs
u_\eps(\cdot) $ and $\omega_\eps$ are uniformly bounded with respect
to $\eps \in {]0,1]}$ in
$ W^{1,r_*}\big(0,T;(W^{1,\sigma_*}(\Omega))^*\big) \subset C^{1/r_*}\big([0,T]; (W^{1,\sigma_*}(\Omega))^*\big)$. Thus, we have
uniform convergence and obtain $(\bs u,\omega) 
\in C^{1/r_*}\big([0,T]; (\bs
W^{1,\sigma_*}(\Omega))^*{\times}(W^{1,\sigma_*}(\Omega))^*
\big)$. Together with the essential boundedness of $(\bs u,\omega)$ in
$L^2(\Omega){\times} L^2(\Omega)$ this implies 
\[
(\bs u,\omega)\in C_{\mathrm{w}}  ([0,T];\bs L^2(\Omega){\times} L^2(\Omega)).
\]
Hence \eqref{3.3} is established. 
Moreover, with \eqref{4.8} and the uniform convergence  we deduce
the initial conditions \eqref{3.10}, i.e.\ $\bs u(\cdot,0)=\bs u_0$ and
$\omega(\cdot,0)= \omega_0$.

\noindent\underline{Step 6. Energy estimates:} 
To obtain the energy-dissipation inequality \eqref{eq:NS.Ener} for the
Navier-Stokes equation, we insert $\bs{w}=\bs{u}_\eps(t)$ into \eqref{4.15},
integrate over the interval $[0,t]$, drop the non-negative term
$\int_0^t\!\int_\Omega \eps |\bs D(\bs u_\eps)|^r \dd x \dd t$, and take the
limit $\eps\to 0$.

Finally, we insert $z\equiv 1$ into \eqref{4.17}, integrate over $[0,t]$ and
add this identity to the one just obtained for $\bs{u}_\eps$. Using
$ \frac{k_\eps}{\eps+\omega_\eps} - \frac{k_\eps}{\eps+\omega_\eps+\eps k_\eps}
\geq 0$ we can drop the two dissipation terms involving
$|\bs D(\bs u_\eps)|^2$. Moreover, the regularization term
$\int_\Omega \eps |\nabla k_\eps|^{r-2} \nabla k_\eps \cdot \nabla z \dd x $
with $z\equiv 1$ gives $0$.  Hence, taking the limit $\eps\to 0$ yields
inequality \eqref{eq:EstimTotEner} for the total energy.

With this, the proof of our main existence result in Theorem
\ref{th:MainExist} is complete.

\appendix

\section{Appendix.\ Existence of approximate solutions}
\label{s.A}

We now provide the proof of Proposition \ref{p4.1}, which will be
obtained as an application of a general existence result of evolutionary
equations of pseudo-monotone type. 

We consider a separable reflexive Banach space $\bfV$ that is
continuously and densely embedded in a Hilbert space $\bfH$ such that
$\bfV \subset \bfH \approx \bfH^* \subset \bfV^*$. For $U\in \bfV$ and
$\Xi \in \bfV^*$ we denote the dual pairing by $\langle
\Xi, U\rangle$.  Our operator
$A:\bfV\to \bfV^*$ is assumed to satisfy the following conditions
depending on $p>1$: 
\begin{subequations}
  \label{eq:A.Assum}
\begin{align}
&  \label{eq:A.Assum.a}
\text{$p$-boundedness:}&& \exists\, C_1>0: \ \|A(U)\|_{\bfV^*}\leq
C_1\big(1{+}\|U\|_\bfV^{p-1}\big) \ \text{ for all }U\in \bfV;
\\ 
& \label{eq:A.Assum.b}
\text{$p$-coercivity:} && \exists \,C_2>0: \ \langle A(U),U\rangle \geq
\frac1{C_2} \|U\|_\bfV^{p}-C_2  \quad \;  \text{ for all } U\in \bfV;
\\[0.4em]
&  \label{eq:A.Assum.c}
\text{pseudo-monotonicity:}\hspace*{-0.6em} &&\left\{
\begin{aligned} &\text{if }U_m\weak U \text{ in }\bfV
\text{ and } \limsup_{m\to\infty} \langle A(U_m),U_m{-}U\rangle \leq
0, \text{ then}\\
& \  \langle A(U),U{-}V\rangle \leq \liminf_{m\to\infty}
\langle A(U_m),U_m{-}V\rangle    \text{ for all }V\in \bfV.
\end{aligned}\right\}
\end{align}
\end{subequations}

Under these conditions the following existence result is available.

\begin{theorem}[see e.g.\ {\cite[Thm.\,8.9]{Roub13NPDE}}]
\label{thm:A} Let $\bfV$ and $\bfH$
  be as above and let the operator $A:\bfV\to \bfV^*$ satisfy the
  assumptions \eqref{eq:A.Assum} with $p>1$. Then, for all $T>0$, all $u_0\in
  \bfH$, and all $f \in L^{p'}([0,T];\bfV^*)$ there exists a solution
  $u \in L^p(0,T;\bfV)\cap C([0,T];\bfH)\cap W^{1,p'}(0,T;\bfV^*)$ of
  the Cauchy problem 
\begin{equation}
\label{eq:AbstrCP}
u'(t) + A(u(t)) = f(t) \ \text{ in }\bfV^* \text{ for a.a.\ } 
t\in [0,T] \qquad \text{and} \qquad u(0)=u_0.
\end{equation}
\end{theorem} 

To apply this result we choose $p=r>3$, $U=(\bs u,\omega,k)$, 
\[
\bfH = \bs L^2_{\diver}(\Omega) \times L^2(\Omega) \times
L^2(\Omega), \quad \text{and} \quad \bfV = 
\bs W^{1,r}_\text{per,div}(\Omega) \times W^{1,r}_\text{per}(\Omega)
\times W^{1,r}_\text{per}(\Omega).
\]
The operator $A$ is defined to make the approximate
system  \eqref{eq:ApproxSys} equivalent to the abstract Cauchy problem
\eqref{eq:AbstrCP}. We recall that $\eps>0$ is fixed in Proposition
\ref{p4.1}, so we do not keep track of the dependence on $\eps$. With
$V=(\bs v,\varphi,w)$ we define $A:V\to V^* $ by 
\begin{align}
 \label{A1.4}
 & \langle A(U),V\rangle= I(U,V)\nonumber\\
 & :=\int_\Omega \bs u{\cdot} \nabla \bs{u}\cdot\bs{v}
 +\int_\Omega\frac{k^+}{\eps+\omega^+}\,\bs{D}(\bs{u}):\bs{D}(\bs{v})
\\
& \quad
 +\int_\Omega\varphi\bs{u}\cdot\nabla\omega 
 +\int_\Omega\frac{k^+}{\eps+\omega^+}\,\nabla\omega\cdot\nabla\varphi
 +\int_\Omega\omega^+\omega\varphi\nonumber
\\
& \quad
 +\int_\Omega w\bs{u}\cdot\nabla k
 +\int_\Omega\frac{k^+}{\eps{+}\omega^+}\,\nabla k\cdot\nabla w
 -\int_\Omega\frac{k^+}{\eps{+}\omega^+ {+} \eps k^+}\, 
     \big|\bs{D}(\bs{u})\big|^2 w
 +\int_\Omega k^+\omega^+w
\nonumber
\\
& \quad
+\eps\int_\Omega\Big(\big|\bs{D}(\bs{u})\big|^{r-2}\bs{D}(\bs{u}):\bs{D}(\bs{v})+|\bs{u}|^{r-2}\bs{u}\cdot\bs{v}\nonumber\\
&  \quad\qquad\qquad 
+|\nabla\omega|^{r-2}\nabla\omega\cdot\nabla\varphi+|\omega|^{r-2}\omega\varphi+|\nabla k|^{r-2}\nabla k\cdot\nabla w+|k|^{r-2}kw\Big).\nonumber
\end{align}
For the rest of this appendix we continue to omit the measure symbol
``$\dd x$'' for integration over $\Omega$. Moreover we have set
$\alpha_2=\nu_0=1$ for notational simplicity, because these numerical
constant have no influence on the analysis. 

\begin{proof}[Proof of Proposition  \ref{p4.1}] 
It remains to establish the conditions \eqref{eq:A.Assum} on the
operator $A$. 
\medskip

\underline{Step 1. $r$-boundedness \eqref{eq:A.Assum.a}:} Using $r>3$ and
H\"older's inequality, it is easily seen that all integrals in the
definition of $I(U,V)$ are well defined. In particular, we find a
constant $c_1>0$ such  that 
\begin{equation}
 \label{A1.5}
 \big|I(U,V)\big| \leq c_1\big( \|U\|_\bfV^2+\|U\|_\bfV^{r-1}\big) 
   \|V\|_\bfV \ \text{ for all }U,V\in \bfV. 
\end{equation}
But this implies \eqref{eq:A.Assum.a} because of $r\geq 3$.  
\medskip

\underline{Step 2. $r$-coercivity \eqref{eq:A.Assum.b}:} For estimating
$\langle A(U),U\rangle =I(U,U)$ from below we see that all convective
terms disappear because of $\diver \bs u=0$. After dropping the three
non-negative terms arising from the dissipation terms involving
$k^+/(\eps{+}\omega^+)$ we find
\begin{align}
\label{A1.12}
&  \langle A(U),U\rangle=I(U,U) \geq 
\eps\big\|(\bs{D}(\bs{u}),\bs{u},\nabla\omega,\omega, \nabla k,
k)\big\|_{L^r(\Omega)}^r  -\int_\Omega\frac{k^+}{\eps{+}\omega^+ {+} \eps k^+} 
\, \big|\bs{D}(\bs{u})\big|^2 k
\end{align}
for all $U \in \bfV$. 
We now use $k^+/(\eps{+}\omega^+{+}\eps k^+)\leq 1/\eps$  and $r\gneqq
3$. By H\"older's and Young's inequality we find
$c_2>0$ such that 
\[
 \int_\Omega \frac{k^+}{\eps{+}\omega^+{+}\eps
   k^+}\,\big|\bs{D}(\bs{u})\big|^2k 
 \leq 
 \frac1\eps \int_\Omega \big|\bs{D}(\bs{u})\big|^2k 
\leq 
\frac\eps 2\int_\Omega\big|\bs{D}(\bs{u})\big|^r+\frac\eps
2\int_\Omega |k|^r+c_2,
\]
where the constant $c_2$ depends on $\eps>0$, $r>3$, and
$\operatorname{vol}(\Omega)$. Inserting this into \eqref{A1.12} and
using Korn's inequality in $\bs W^{1,r}(\Omega)$ we have established
\eqref{eq:A.Assum.b} for $p=r$.  
\medskip

\underline{Step 3. Strong convergence:} In the remaining two 
steps we consider a sequence $U_m=(\bs u_m,\omega_m,k_m)$ satisfying
the assumptions in condition \eqref{eq:A.Assum.c}, namely
\begin{align}
\label{eq:PsMo.Ass}
\text{(a) } \ U_m \weak U \ \text{ in } \bfV \qquad \qquad 
\text{(b) } \limsup_{m\to \infty} \:\langle A(U_m),U_m{-}U\rangle \leq 0.
\end{align}
In this step we first show that this implies the strong convergence
$U_m\to U$ in $\bfV$, and in Step 4 we deduce the liminf estimate for
\eqref{eq:A.Assum.c}. 

Combining parts (a) and (b) of \eqref{eq:PsMo.Ass} we immediately
obtain 
\begin{align}
\label{eq:Duality}
 \limsup_{m\to \infty} \big\langle A(U_m)-A(U)\,, \, U_m-U\big\rangle
\leq 0 . 
\end{align}
We decompose these duality products into ten separate integrals,
namely  
\begin{align}
 \label{A1.16}
&\big\langle A(U_m)-A(U) , U_m-U\big\rangle
=\sum_{j=1}^{10} K_{j,m} \\
\nonumber &
:=\int_\Omega\big[\bs u_{m}{\cdot}\nabla \bs u_m {-} \bs u{\cdot}\nabla
          \bs{u}\big]\cdot(\bs{u}_m{-}\bs{u})
+\!\int_\Omega\Big[\frac{k_m^+}{\eps{+}\omega_m^+}\,\bs{D}(\bs{u}_m)- 
  \frac{k^+}{\eps{+}\omega^+}\,\bs{D}(\bs{u})\Big] {:}
   \bs{D}(\bs{u}_m{-}\bs{u})\nonumber\\ 
 &\quad
 +\int_\Omega(\bs{u}_m{\cdot}\nabla\omega_m-\bs{u}{\cdot}\nabla\omega)
 \,(\omega_m{-}\omega)
+\int_\Omega\Big[\frac{k_m^+} {\eps{+}\omega_m^+} \, 
 \nabla\omega_m-\frac{k^+}{\eps{+}\omega^+}\, \nabla\omega\Big]
 \cdot\nabla(\omega_m{-}\omega)\nonumber\\ 
&\quad 
  +\int_\Omega(\omega_m^+\omega_m-\omega^+\omega)(\omega_m{-}\omega)
  +\int_\Omega(\bs{u}_m\cdot\nabla k_m-\bs{u}\cdot\nabla k)(k_m{-}k)\nonumber\\
&\quad
  +\int_\Omega\Big[\frac{k_m^+}{\eps{+}\omega_m^+}\,\nabla
    k_m-\frac{k^+}{\eps{+}\omega^+}\,\nabla k\Big]\cdot\nabla(k_m{-}k)
  +\int_\Omega(k_m\omega_m^+-k\omega^+)(k_m{-}k)\nonumber
\\
 &\quad 
  -\int_\Omega\Big(\frac{k_m^+}{\eps{+}\omega_m^+{+}\eps k_m^+}\, 
   \big|\bs{D}(\bs{u}_m)\big|^2-\frac{k^+}{\eps{+}\omega^+{+}\eps
     k^+}\,\big|\bs{D}(\bs{u})\big|^2\Big)(k_m-k)\nonumber\\
 &\quad +\int_\Omega\eps\bigg[\big( \Phi_r(\bs{D}(\bs{u}_m))
 - \Phi_r(\bs D(\bs u)) \big){:}\bs D(u_m{-}\bs u) +
  \big(\Phi_r(\bs u_m) - \Phi_r(\bs u)\big) {\cdot}(\bs u_m{-}\bs u)\nonumber
 \\
&\qquad\qquad\quad  
+\big(\Phi_r(\nabla\omega_m)-\Phi_r(\nabla\omega)\big) \cdot
  \nabla(\omega_m{-}\omega)+
  \big(\Phi_r(\omega_m)-\Phi_r(\omega)\big)(\omega_m{-}\omega)\nonumber
\\
&\qquad\qquad\quad  
+\big(\Phi_r(\nabla k_m)-\Phi_r(\nabla k)\big)\cdot
  \nabla(k_m{-}k)+\,
  \big(\Phi_r(k_m)-\Phi_r(k)\big)\,(k_m{-}k) \; \bigg],\nonumber
\end{align}
where $\Phi_r(\bs \xi):=|\bs \xi|^{r-2}\bs \xi$. The last term $K_{10,m}$
can be used to control $U_m-U$ in the norm of $\bfV$ by 
using the estimate 
\[
 \big(\Phi_r(\bs\xi)-\Phi_r(\bs\eta)\big)\cdot(\bs\xi{-}\bs\eta)\geq 
2^{2-r}\big|\bs\xi-\bs\eta\big|^r\quad\text{for all }\bs\xi,\bs\eta 
\in\mathbb{R}^N, 
\]
see \cite{Ref17} for the derivation of the exact constant. 
In particular, we find 
\begin{align}
\label{eq:K10}
 K_{10,m} \geq \eps 2^{2-r} \big\| U_m- U\big\|_{\bfV}^r ,
\end{align}
and the strong convergence $U_m\to U$ follows if we show $\limsup_{m\to \infty}
K_{10,m}\leq 0$.  

By \eqref{eq:Duality} we control the limsup of $\sum_1^{10} K_{j,m}$
and hence obtain 
\begin{align*}
\limsup_{m\to \infty} K_{10,m}& =
\limsup_{m\to \infty} \Big(\sum_{j=1}^{10}K_{j,m} -
\sum_{l=1}^9K_{l,m} \Big) \\
& \leq \limsup_{m\to
  \infty}\sum_{j=1}^{10}K_{j,m} - \liminf_{m\to \infty}
\sum_{l=1}^9K_{l,m} 
\overset{\text{\eqref{eq:Duality}}}\leq 0 \ -\ \sum_{l=1}^9  \liminf_{m\to
  \infty} K_{l,m} .
\end{align*}
Thus, it suffices to show $\liminf_{m\to \infty} K_{l,m}\geq 0 $ for
$l\in \{1,...,9\}$. To do so, we use $U_m \weak U$ (i.e.\
(\ref{eq:PsMo.Ass}a)), which by $r>3$ and the compact embedding
$W^{1,r}(\Omega) \Subset C^0(\overline\Omega)$ implies
\begin{align}
\label{eq:UniformCvg}
\bs u_m \to \bs u,\quad \omega_m\to \omega, \quad k_m \to k\quad 
\text{uniformly in } \overline\Omega.
\end{align}

For treating $K_{1,m}$ we use integration by parts and $\diver \bs
u_m=\diver \bs u=0$ to find 
\[
K_{1,m}=\int_\Omega \big( \diver(\bs u_m{\otimes} \bs u_m) : \nabla
\bs u - \bs u{\cdot}\nabla \bs u \cdot \bs u_m\big) \ \to \
\int_\Omega \big( \diver(\bs u{\otimes} \bs u) : \nabla \bs u - \bs
u{\cdot}\nabla \bs u \cdot \bs u\big) \ = \ 0,
\]
because of the uniform convergence $\bs u_m \to \bs u$.

Similarly, the other convective terms $K_{3,m}$ and $K_{6.m}$
converge to $0$, since $\omega_m\to \omega$ and $k_m \to k $ converge
uniformly.

For the second term $K_{2,m}$ we again use the uniform convergence in
the decomposition
\[ 
K_{2,m}=
\int_\Omega
\big(\frac{k_m^+}{\eps{+}\omega_m^+}-\frac{k^+}{\eps{+}\omega^+}\big)    
\,\bs{D}(\bs{u}_m) :\bs{D}(\bs{u}_m{-}\bs{u})
+
\int_\Omega\frac{k^+}{\eps{+}\omega^+}\, \bs{D}(\bs{u}_m{-}\bs{u}) 
   :\bs{D}(\bs{u}_m{-}\bs{u}).
\]
The first integral converges to $0$ as the two terms involving $\bs D$
are bounded in $\bs L^r(\Omega)\subset \bs L^2(\Omega)$ while the
prefactor converges to $0$ uniformly. The second integral is
non-negative, hence $\liminf\limits_{m\to \infty} K_{2,m}\geq 0 $ follows. 
Analogously, the \,$\liminf\limits_{m\to\infty}$ \,of $K_{4,m}$ and
$K_{7,m}$ is non-negative. 

By uniform convergence of the integrands we easily obtain $K_{5,m}\to
0$ and $K_{8,m}\to 0$. 

In $K_{9,m}$ the integrand is a product of a function bounded
uniformly in $L^{r/2}(\Omega)$ and $k_m{-}k$, which converges
uniformly to $0$; hence $K_{9,m}\to 0$ as well. 

This finishes the proof of Step 3 guaranteeing $U_m\to U$ in
$\bfV$.  \medskip

\underline{Step 4. $A$ is pseudo-monotone:} For the sequence $U_m$
satisfying \eqref{eq:PsMo.Ass}  we have to show
\begin{align}
  \label{eq:DualityLiminf}
 \langle A(U),U{-}V\rangle \leq \liminf_{m\to\infty}
\langle A(U_m),U_m{-}V\rangle  \text{ for all }V=(\bs v,
\varphi,w)\in \bfV 
\end{align}
By Step 3 we are now able to use the strong
convergence $U_m\to U$. 
Again we split the duality-product term into ten parts and treat the
parts separately:  \vspace{-0.4em}
\begin{align}
 \label{A1.18}
 & \big\langle {A}(U_m),U_m-V\big\rangle = \sum_{j=1}^{10}
 G_{j,m}\\[-0.4em] \nonumber
& =:\int_\Omega \bs u_m{\cdot} \nabla \bs{u}_m\cdot(\bs{u}_m{-}\bs{v})
 +\int_\Omega\frac{k_m^+}{\eps{+}\omega_m^+}\,\bs{D}(\bs{u}_m): 
       \bs{D}(\bs{u}_m{-}\bs{v})
\\ \nonumber
& \quad 
 +\int_\Omega\bs{u}_m{\cdot}\nabla\omega_m \, (\omega_m{-}\varphi)
 +\int_\Omega\frac{k_m^+}{\eps{+}\omega_m^+}\,\nabla\omega_m\cdot 
     \nabla(\omega_m{-}\varphi)\nonumber\\
 & \quad
 +\int_\Omega\omega_m^+\omega_m(\omega_m{-}\varphi)
 +\int_\Omega \bs{u}_m{\cdot}\nabla k_m\,(k_m{-}w)
 +\int_\Omega\frac{k_m^+}{\eps{+}\omega_m^+}\,\nabla k_m\cdot 
   \nabla(k_m{-}w) 
\nonumber\\ & \quad
  +\int_\Omega k_m\omega_m^+(k_m{-}w) 
 -\int_\Omega\frac{k_m^+}{\eps{+}\omega_m^+{+}\eps k_m^+}\,\big|\bs{D}(\bs{u}_m)\big|^2(k_m{-}w)\nonumber\\
 & \quad
 +\int_\Omega
 \eps\Big(\Phi_r(\bs{D}(\bs{u}_m)):\bs{D}(\bs{u}_m{-}\bs{v}) 
  +\Phi_r(\bs{u}_m)\cdot(\bs{u}_m{-}\bs{v})
  +\Phi_r(\nabla\omega_m)\cdot\nabla(\omega_m{-}\varphi) \nonumber
\\ & \qquad\qquad \qquad 
  +\Phi_r(\omega_m)(\omega_m{-}\varphi)
  +\Phi_r(\nabla k_m)\cdot\nabla(k_m{-}w)+\Phi_r(k_m)(k_m{-}w)\Big).\nonumber
\end{align}
Using the uniform convergence of $U_m$ (see \eqref{eq:UniformCvg}) and
the strong convergence in $L^r(\Omega)$ of the derivatives $\nabla
U_m$ it is straight forward to see that the integrals $G_{j,m}$ for
$j\in \{1,...,9\}$ converge to their respective limits. For $G_{10,m}$
we can use the estimate
\[
 \big|\Phi_r(\bs\xi)-\Phi_r(\bs\eta) \big| \leq 3r
 \big(|\bs\xi|+|\bs\eta|\big)^{r-2}\big|\bs\xi -\bs\eta \big| \quad
\text{for all }\bs\xi,\bs\eta \in \mathbb R^N, 
 \] 
see \cite[exerc.\,10.a, p.\,257]{Ref4}. Thus, we conclude 
that \eqref{eq:DualityLiminf} holds, even with equality.

Hence, all the assumptions in \eqref{eq:A.Assum} are established,
Theorem \ref{thm:A} is applicable, and the proof of
Proposition  \ref{p4.1} is complete. 
\end{proof}

\begin{remark}\label{rem:OldProof}
An alternative proof for Proposition \ref{p4.1} is given in the first
draft \cite{MieNau18arXiv} of the present work. That proof is based on the
method of elliptic regularization of abstract evolution equations,
cf.\,\cite[Ch.\,3, Thm.\,1.2]{Ref18}.\vspace{-0.3em}
\end{remark}

\paragraph*{Acknowledgments.} 
The authors are indebted to H. Z. Baumert (IAMARIS) for
helpful discussions on physical aspects of Kolmogorov's two-equation model
of turbulence. A.M. was partially supported by Deutsche Forschungsgemeinschaft
(DFG) via the project A5 ``Pattern formation in coupled parabolic
systems'' of the Collaborative Research Center SFB\,910 \emph{Control of
Complex Nonlinear Systems}, project no.\,163436311. The authors are grateful to
anonymous referees for useful remarks which led to improvements upon
earlier versions of the paper.
\vspace{-0.3em}





\allowdisplaybreaks[4]

\end{document}